\documentclass[11pt,a4paper]{article}
\usepackage[utf8x]{inputenc}
\usepackage{ucs}
\usepackage{amsmath}
\usepackage{amsfonts}
\usepackage{amssymb}
\usepackage{amsthm}
\usepackage[width=16.00cm, height=23.00cm]{geometry}
\usepackage{tikz-cd}

\usepackage{color}
\usepackage{mathptmx}
\usepackage[T1]{fontenc}
\usepackage[shortlabels]{enumitem}

\newtheorem{theorem}{Theorem}[section]
\newtheorem{lemma}[theorem]{Lemma}
\newtheorem{proposition}[theorem]{Proposition}
\newtheorem{corollary}[theorem]{Corollary}
\theoremstyle{definition}
\newtheorem{remark}[theorem]{Remark}

\newtheorem{example}[theorem]{Example}
\usepackage{hyperref}

\DeclareMathOperator{\ext}{ext}
\DeclareMathOperator{\re}{Re}

\DeclareMathOperator{\proj}{proj}
\DeclareMathOperator{\spanned}{span}

\newcommand{\HRe}{H_{\infty}^{\lambda}[\re >0]}
\newcommand{\HRep}{H_{\infty,+}^{\lambda}[\re >0]}
\newcommand{\HRek}[1]{H_{\infty}^{\lambda}[\re > 1/#1]}
\newcommand{\HRevv}[1]{H_{\infty}^{\lambda}([\re > 0], #1)}
\newcommand{\punkt}{\,\begin{picture}(-1,1)(-1,-3)\circle*{2.5}\end{picture}\;\,}

\newcommand\restrict[1]{\raisebox{-.5ex}{$\vert$}_{#1}}

\allowdisplaybreaks

\title{Fr\'echet spaces of general Dirichlet series}

\author{Andreas Defant\thanks{Institut f\"ur Mathematik, Carl von Ossietzky Universit\"at, 26111 Oldenburg, Germany (defant@mathematik.uni-oldenburg.de).Partially supported by MINECO and FEDER project MTM2017-83262-C2-1-P}%
\and %
Tom\'as Fern\'andez-Vidal\thanks{Departamento de Matem\'{a}tica,
Facultad de Cs. Exactas y Naturales, Universidad de Buenos Aires and IMAS-CONICET. Ciudad Universitaria, Pabell\'on I (C1428EGA) C.A.B.A., Argentina (tfernandezvidal@yahoo.com.ar). Supported by PICT 2015-2299} %
\and %
Ingo Schoolmann\thanks{Institut f\"ur Mathematik, Carl von Ossietzky Universit\"at, 26111 Oldenburg, Germany (ingo.schoolmann@uni-oldenburg.de)}%
\and %
Pablo Sevilla-Peris\thanks{Institut Universitari de Matem\`atica Pura i Aplicada,
Universitat Polit\`{e}cnica de Val\`encia, cmno Vera s/n, 46022,
Val\`encia, Spain (psevilla@mat.upv.es). Supported by MINECO and FEDER project MTM2017-83262-C2-1-P}}

\date{}

\begin{document}

\maketitle

\begin{abstract}
Inspired by  a recent article  on  Fr\'echet spaces of  ordinary Dirichlet series
$\sum a_n n^{-s}$ due to J.~Bonet,
we study  topological and geometrical properties  of certain scales of Fr\'echet spaces of general Dirichlet spaces
$\sum a_n e^{-\lambda_n s}$. More precisely, fixing a frequency $\lambda = (\lambda_n)$, we focus on the  Fr\'echet space of  $\lambda$-Dirichlet series
which have limit functions  bounded on all half planes strictly smaller than the  right half plane $[\re >0]$. We develop an abstract setting
of pre-Fr\'echet spaces of $\lambda$-Dirichlet series generated by certain admissible normed spaces of
$\lambda$-Dirichlet series and the abscissas of convergence they generate, which allows also to define  Fr\'echet spaces of $\lambda$-Dirichlet series for which  $a_n e^{-\lambda_n/k}$ for each $k$ equals  the Fourier coefficients of a function on an appropriate  $\lambda$-Dirichlet group.
\end{abstract}

\section{Introduction} \label{intro}

Given a frequency $\lambda = (\lambda_n)$, i.e. a strictly increasing unbounded sequence of non-negative real numbers,
a $\lambda$-Dirichlet series  is a (formal) series of the form $D=\sum a_{n}e^{-\lambda_{n}s}$, where $s$ is a complex variable and the $a_n \in \mathbb{C}$ the Dirichlet coefficients.
It is a well known fact that general Dirichlet series naturally converge on half planes $[\re > \sigma]$, and there they define holomorphic functions (see \cite[Theorem~2]{hardy2013general} or \cite[Lemma~4.1.1]{queffelec2013diophantine}).\\
The study of these series has a long history initiated at the beginning of the 20th century by  prominent mathematicians like H.~Bohr, G.H.~Hardy, and M.~Riesz, among others. One of their main contributions was the study of the analytic properties of the functions defined by general Dirichlet series. The most important example of a frequency
is certainly  given by $\lambda = (\log n)$, leading to  ordinary Dirichlet series $\sum a_n n^{-s}$, which play a fundamental role in analytic number theory. \\

In recent years there has been a revival of interest in the interplay between analysis and Dirichlet series opened up by those early contributions. This
`modern theory of Dirichlet series' mainly focuses on the study of  ordinary series, which  involves the intertwining
of classical work with modern analysis -- like functional analysis, harmonic analysis, infinite dimensional holomorphy, probability theory, as well as analytic number theory. \\
The space of Dirichlet series that define a bounded holomorphic function on $[\re >0]$ plays a major role within this modern approach. Bonet in  \cite{bonet2018frechet} defined and studied the Fr\'echet space of all (ordinary) Dirichlet
series which converge (and hence define a holomorphic function) on $[\re >0]$ and are bounded on every smaller half plane $[\re > \sigma]$ for $\sigma >0$. Given a frequency $\lambda$, the space $\mathcal{D}_{\infty}(\lambda)$ of all
$\lambda$-Dirichlet series that define a bounded holomorphic function on $[\re >0]$ was defined in \cite{schoolmann2018bohr}. Inspired by the work of Bonet, in this article we focus on the space
$\mathcal{D}_{\infty,+}(\lambda)$ of all $\lambda$-Dirichlet series that on $[\re >0]$ converge to a (then necessarily holomorphic) function which is bounded on each half plane $[\re > \sigma]$ with $\sigma >0$.\\
Carrying its  natural topological structure, this space $\mathcal{D}_{\infty,+}(\lambda)$  is a pre-Fr\'echet space which (as we will see) in general fails to be a Fr\'echet space. Bonet proved in \cite{bonet2018frechet} that the topological structure of $\mathcal{D}_{\infty,+}((\log n))$  is rich. It is a  Fr\'echet  algebra
which is a  Schwartz space, and the monomials  $(n^{-s})$ form a  Schauder basis, but which is not nuclear.
His  proofs combine modern techniques from the theory of ordinary Dirichlet series like Bohr's inequality or Bayart's Montel theorem with classical results on K\"othe sequence spaces like the Grothendieck-Pietsch test for nuclearity.\\

Making the jump from the frequency $(\log n)$ to an arbitrary frequency reveals challenging consequences. For example, much of the theory for ordinary series relies on `Bohr's theorem', which in particular implies that  each
ordinary Dirichlet series which converges to a bounded function on some half plane
$[\re > \sigma]$, in fact converges uniformly on each smaller half plane $[\re > \mu]$ with $\mu > \sigma$. However, for general Dirichlet series, it is known that the  validity of Bohr's theorem depends very much on the `structure' of the frequency (see Section~\ref{sec:equivalence}).\\
As a consequence, for Dirichlet series build over an arbitrary frequency $\lambda$, the general occurrence for $\mathcal{D}_{\infty,+}(\lambda)$ is much  more complex. To illustrate this, consider the frequency $\lambda = (n)$. 
Then, looking at the change of variables $s \in [\re >0] \leftrightsquigarrow z = e^{-s} \in \mathbb{D}$, each Dirichlet series $\sum a_{n} e^{-ns}$ is transformed into a  power series $\sum a_{n} z^{n}$. It turns out that
$\mathcal{D}_{\infty,+}((n))$ is nothing else than  the nuclear Fr\'echet space $H(\mathbb{D})$ of all holomorphic functions $f: \mathbb{D} \to \mathbb{C}$ (with the topology of uniform convergence on compact sets).
In particular,  $\mathcal{D}_{\infty,+}((n))$ is isomorphic to a countable projective limit of Banach spaces, all isometrically equal to the Hardy space $H_\infty(\mathbb{T})$, which relates its study with Fourier analysis on a
compact abelian group.\\
A third natural example of frequency is $\lambda = (\log p_n)$, where $p_n$ stands for the $n$th prime. We show that for this frequency  $\mathcal{D}_{\infty,+}(\lambda)$ is a Fr\'echet space that, by a result of Bohr, may be
identified with a K\"othe echelon space (a projective limit of countably many weighted $\ell_1$-spaces) which, though  Schwartz, again fails to be nuclear. \\

One of our main purposes here is to clarify the situation, studying the structure of the pre-Fr\'echet spaces $\mathcal{D}_{\infty,+}(\lambda)$ depending on the frequency $\lambda$. First of all we see that these are always
Schwartz. Then we focus on the following properties: completeness, barrelledness, Montel, the monomials being a Schauder basis and nuclearity. We show that for $\mathcal{D}_{\infty,+}(\lambda)$ the first three properties are
equivalent, and that they hold if and only if Bohr's theorem holds for $\lambda$ and in this case, the space can be identified with a countable projective limit of certain Hardy spaces on so-called Dirichlet groups and the limit functions defined by the Dirichlet series in $\mathcal{D}_{\infty,+}(\lambda)$ have a natural description in terms of uniformly almost periodic functions on the right half plane. The monomials are a Schauder basis whenever Bohr's theorem holds for $\lambda$. Finally we also characterise those frequencies for which the space
$\mathcal{D}_{\infty,+}(\lambda)$ is nuclear.\\
We present a more general setting, which allows to study various similar types of (pre-)Fr\'echet spaces of general Dirichlet series with similar ideas. Fixing a frequency $\lambda$, the idea is to study  (pre-)Fr\'echet spaces
of $\lambda$-Dirichlet series which are  generated by what we call a `$\lambda$-admissible'
normed space of $\lambda$-Dirichlet series.
This allows to incorporate not only $\mathcal{D}_{\infty,+}(\lambda)$ in our study, but also Hardy-type Fr\'echet spaces generated by $\mathcal{H}_p(\lambda)$ with $1 \leq p \leq \infty$, following what was done in \cite{FVGaMeSe_20}.\\
One of our main tools is the representation of our pre-Fr\'echet spaces as countable projective limits of their natural `Banach space precursors' ($\mathcal{D}_{\infty}(\lambda)$ and $\mathcal{H}_{p}(\lambda)$).
In this sense our article continues  a series of recent articles on general Dirichlet series (see \cite{CaDeMaSc_VV,defant2019hardy,defant2020variants,schoolmann2018bohr,schoolmann2018bohrA}),  which
combine classical results from the deep analysis presented by  Hardy and Riesz in
 \cite{hardy2013general}, with various topics from modern analysis (as complex analysis, functional analysis
in Banach and Fr\'echet spaces, Fourier analysis on $\mathbb{R}$, or harmonic analysis on compact abelian groups). \\

Finally, we remark that the  study  on Fr\'echet spaces of general Dirichlet series  undertaken here, forces us to consider  independently interesting issues within related Banach spaces -- such as the hypercontractivity of
translation operators, or  Montel-type theorems for Banach spaces of uniformly $\lambda$-almost-periodic functions and for Hardy-type  spaces of $\lambda$-Dirichlet series.

\section{Preliminaries}
We collect the basic results on Dirichlet series and Fr\'echet spaces needed in following.

\subsection{Dirichlet series} \label{diri}

We begin with a short account of the basic facts on general Dirichlet series that will be needed along the article. We refer the reader to \cite{defant2018Dirichlet,hardy2013general,queffelec2013diophantine} for the basics on ordinary and general Dirichlet series.\\
Given a frequency $\lambda$, all (formal) $\lambda$-Dirichlet series $\sum a_n e^{-\lambda_n s}$  are denoted by $\mathfrak{D}(\lambda)$.
The following `abscissas' rule the convergence theory of general  Dirichlet series $D=\sum a_{n}e^{-\lambda_{n}s}$:
\begin{gather*}
\sigma_{c}(D)=\inf  \{ \sigma \in \mathbb{R} \colon D \text{ converges on } [\re >\sigma] \},
\\
\sigma_{a}(D)=\inf \{ \sigma \in \mathbb{R} \colon D \text{ converges absolutely on } [\re>\sigma] \},
\\
\sigma_{u}(D)=\inf \{ \sigma \in \mathbb{R}\colon D \text{ converges uniformly on } [\re>\sigma] \},
\\
\sigma_{b}(D)=\inf \{ \sigma\in \mathbb{R}  \colon  D \text{ converges and defines a bounded function on } [\re>\sigma] \} \,.
\end{gather*}
By definition $\sigma_{c}(D)\le \sigma_{b}(D) \le \sigma_{u}(D)\le \sigma_{a}(D)$, and in  general all these abscissas differ. Let us recall once again that a general Dirichlet series $D = \sum a_{n} e^{- \lambda_{n}s}$ defines a holomorphic functions on the half plane $[\re >\sigma_{c}(D)]$. \\

Another important, say geometric, value associated to a frequency $\lambda$ is the maximal width of the strip of convergence and non absolutely convergence, that is
\[
L(\lambda) : =\sup_{D \in \mathfrak{D}(\lambda)} \sigma_{a}(D)-\sigma_{c}(D),
\]
which, as shown by Bohr  in \cite[\S3, Hilfssatz~2 and~3]{BohrBemerkungen}, can be computed as follows
\[
L (\lambda ) = \sigma_{c} \big(\textstyle \sum e^{-\lambda_{n}s} \big)
= \displaystyle \limsup_{n\to \infty} \frac{\log(n)}{\lambda_{n}} \,.
\]

As we already pointed out earlier, fulfilling Bohr's theorem is one of the key properties within the theory. Let us succinctly explain what does that mean. Let   $\mathcal{D}_{\infty}^{\ext}(\lambda)$ denote the space of $\lambda$-Dirichlet series that
converge somewhere, and whose limit function extends  to a bounded holomorphic function on $[\re >0]$. Then the frequency $\lambda$ is said to satisfy `Bohr's theorem' (or that Bohr's theorem holds for $\lambda$) if $\sigma_{u}(D) \leq 0$ for
every $D \in \mathcal{D}_{\infty}^{\ext}(\lambda)$.\\
The question then is to find conditions on $\lambda$ so that this property holds. The first one to address this question was Bohr  (thus explaining the name), who in \cite{Bohr} isolated a concrete sufficient condition which roughly speaking prevents the
$\lambda_n$s from getting too close too fast. More precisely, he showed that if $\lambda$ satisfies what we now call `Bohr's condition':
\begin{equation} \label{BC} \tag{BC}
\exists ~l = l (\lambda) >0 ~ \forall ~\delta >0 ~\exists ~C>0~\forall~ n \in \mathbb{N}: ~~\lambda_{n+1}-\lambda_{n}\ge Ce^{-(l+\delta)\lambda_{n}}\,,
\end{equation}
then it satisfies Bohr's theorem. Note that $\lambda=(\log n)$ satisfies \eqref{BC} with $l=1$ and, then, Bohr's theorem holds for ordinary Dirichlet series. This is one of the   fundamental tools within the theory of ordinary Dirichlet series (see
\cite[Theorem~1.13]{defant2018Dirichlet} or \cite[Theorem~6.2.2]{queffelec2013diophantine}).\\
Later  Landau in \cite{Landau} improved Bohr's result by showing  that the weaker condition
\begin{equation} \label{LC} \tag{LC}
\forall~ \delta>0~ \exists ~C>0 ~\forall~ n \in \mathbb{N} \colon~ \lambda_{n+1}-\lambda_{n}\ge C e^{-e^{\delta \lambda_{n}}}.
\end{equation}
is also sufficient for Bohr's theorem. Observe that \eqref{BC} implies \eqref{LC}, and that  the frequencies $\lambda =( (\log n)^{\alpha})$ satisfy \eqref{LC} for every $\alpha >0$ but for example $\lambda=(\sqrt{\log n})$  (i.e. $\alpha = 1/2$) fails \eqref{BC}.\\
 We know (see e.g. \cite[Remark~4.8.]{schoolmann2018bohr}) that Bohr's theorem holds for $\lambda$ in each of the following `testable' cases:
\begin{itemize}
  \item   $\lambda$ is $\mathbb{Q}$-linearly independent,
 \item $L(\lambda):=\limsup_{n\to \infty} \frac{\log n}{\lambda_{n}}=0$,
 \item  $\lambda$ fulfills \eqref{LC} (and  in particular, if it fulfills \eqref{BC}).
\end{itemize}
Then, Bohr's theorem holds for the frequencies $\lambda = (\log p_{n})$ (because it is $\mathbb{Q}$-l.i.), $\lambda = (n)$ (for which $L(\lambda) =0$) and $\lambda =( (\log n)^{\alpha})$ for $\alpha >0$ (since, as we just mentioned, it satisfies \eqref{LC}). Recently some other sufficient conditions have been found by Bayart \cite{bayart_hyper}.

\subsection{Fr\'echet spaces} \label{prelim}

We collect here some basic definitions and facts on Fr\'echet that we need all along this article -- all results mentioned are included in the monographs \cite{FlWl68,jarchow2012locally,meise1997introduction}.\\

Let $E$ be vector space and  $\mathcal{P}$  a family of seminorms satisfying the following two conditions: first, for every $x \in E$ there is $p \in \mathcal{P}$ so that $p(x) \neq 0$ and, second,  that for  all
$p_{1}, p_{2} \in \mathcal{P}$ there is some $c>0$ and $p \in \mathcal{P}$ with $\max (p_{1} (x), p_{2} (x)) \leq c p (x)$ for every  $x \in E$. Then the pair $(E,\mathcal{P})$ defines a (locally convex Hausdorff) topology  on $E$ in the following way. A set $O \subset E$
is open whenever for each $x \in O$ there are  $p \in \mathcal{P}$ and  $\varepsilon >0$ so that
$\{x \in E\colon p(x)< \varepsilon \} \subseteq O$.\\
A net $(x_\alpha)$ in $(E, \mathcal{P})$ is Cauchy if for each $p \in \mathcal{P}$ and each $\varepsilon >0$ there is some $\alpha_0$ such that for all
$\alpha_1,\alpha_2 > \alpha_0$ we have $p(x_{\alpha_1}- x_{\alpha_2}) < \varepsilon$. A locally convex space is said to be complete if every Cauchy net in $E$ is convergent.\\

For each seminorm  $p \in \mathcal{P}$ we consider the normed space $(E_p, \|\punkt\|_p)$ given by
\[
E_p:= E/\ker p \, \text{ and } \, \big\|x + \ker p \big\|_p: = p(x) \,,
\]
and  for all $p, q \in \mathcal{P}$ for which there is some $c>0$ such that $q \leq c p$, we may define the (so-called) linking maps
\[
\pi_{p,q}: E_p \to E_q\, \text{ by } x + \ker p   \mapsto x + \ker q\,,
\]
which are all linear with norm $\leq c$.\\
Then, $E$ is called Schwartz (resp. nuclear) if for every
$q \in \mathcal{P}$ there are $p \in \mathcal{P}$  and $c >0$ with $q\leq cp$ such that
$\pi_{p,q}: E_p \to E_q$ is precompact (resp. nuclear). Recall that a (bounded, linear) operator
$u:X \to Y$  between  normed spaces is precompact whenever $u$ maps the unit ball of $X$ into a precompact set of $Y$, and  it is is nuclear whenever there are sequences
$(x_n^\ast)$ in $X^\ast$ and $(y_n)$ in $Y$ such that $\sum_n \|x_n^\ast\| \,\|y_n\| < \infty$ and  $u(x) = \sum_n x_n^\ast(x) y_n$ for all $x \in X$.\\

A locally convex space $(E, \mathcal{P})$ is
\begin{itemize}
\item barrelled if every barrel (i.e., every absolutely convex, closed, and absorbing set in $E$) is a zero-neighbourhood, or equivalently, if it satisfies the uniform boundedness principle (every pointwise bounded set of continuous operators  from $E$ into some  locally convex space $F$ is equicontinuous).

\item  semi Montel if every bounded set is relatively compact.

\item Montel if it is barrelled and semi-Montel.

\item pre-Fr\'echet if $\mathcal{P} = \{ p_k\colon k \in \mathbb{N}\}$ is countable, and in this case
we may assume without loss of generality that  the seminorms are increasing.

\item Fr\'echet if it is pre-Fr\'echet and complete.
\end{itemize}
A standard argument shows that the locally convex topology of a pre-Fr\'echet space is given by a translation invariant metric. It is important to note that every Fr\'echet space is barreled, and that Fr\'echet-Schwartz spaces are Montel.\\

 A sequence $(e_{n})_{n}$ in a locally convex  space $E$ is a Schauder basis if for every $x \in E$ there is a unique sequence $( \alpha_{n} )_{n}$ of scalars such  that $x = \sum_{n=1}^{\infty} \alpha_{n} e_{n}$. In
this case all coefficient functionals $e_n^\ast$ defined by $e_n^\ast(x) = \alpha_n$ are continuous.
If  $(e_{n})$ is a Schauder basis of a Fr\'echet (or more generally barreled) space $(E, \mathcal{P})$, then for every $p \in \mathcal{P}$ there is $q \in \mathcal{P}$ and a constant $C>0$ such that for every $M\ge N$ and every complex sequences $(\alpha_{n})$
\begin{equation} \label{basisineq}
p\big(\sum_{n=1}^{N} \alpha_{n} e_{n}\big)\le C \,q\big(\sum_{n=1}^{M} \alpha_{n}e_{n} \big).
\end{equation}

Let $(X_{k})_{k}$ be a (countable) family of normed spaces and, for each $k$ consider a bounded linear operator $i_{k}: X_{k+1} \to X_k$. Then the pair
\[
\big( X_k, i_{k}\big)_{k\in\mathbb{N}}
\]
is called a countable projective spectrum. The projective limit $\proj X_k$ is defined to be the topological subspace of $\prod_k X_k$ consisting of those $(x_k)$ so that $i_k(x_{k+1})= x_k$ for every $k$. If we denote by $\pi_{n}$ the
canonical projection from $\proj X_{k}$ to $X_{n}$, then
\begin{equation} \label{clementi}
p_{n} (x) = \max_{1 \leq m \leq n}\|\pi_m(x)\|_{X_m}
\end{equation}
defines a seminorm on $\proj X_{k}$. It is easy to see that the collection of all these seminorms  generates a locally convex topology on $\proj X_{k}$ that coincides
with the one induced by $\prod X_{k}$. Hence the projective limit of countably many normed spaces is always a pre-Fr\'echet space. If every $X_{k}$ is Banach, then $\prod_k X_k$ is complete and, then so also is the closed
subspace $\proj X_k$ (then a Fr\'echet space). \\
To see an example, we recall that a real matrix $A = (a_{jk})_{j,k=1}^\infty$ is said to be a (positive) K\"othe matrix, whenever  $0 < a_{jk} < a_{j, k+1}$ for all $k,j$. Then, given   $1 \leq p < \infty$, each of the weighted
$\ell_p$-spaces
\[
\ell_p( (a_{j,k})_{j=1}^\infty) = \Big\{  x \in \mathbb{C}^{\mathbb{N}}
\colon \|x\|_k = \Big( \sum_{j=1}^\infty |a_{jk} x_j|^p  \Big)^{\frac{1}{p}} < \infty
\Big\}\,, \,\,\,k \in \mathbb{N}.
\]
is  obviously isometrically isomorphic to $\ell_p$. Together with the canonical inclusions these form a countable projective spectrum which defines the Fr\'echet space
\begin{equation} \label{Koethe}
\ell_{p} (A) = \proj\ell_p( (a_{j,k})_{j=1}^\infty) = \ \Big\{  x \in \mathbb{C}^{\mathbb{N}}
\colon \|x\|_k = \Big( \sum_{j=1}^\infty |a_{jk} x_j|^p  \Big)^{\frac{1}{p}} < \infty
\,\, \,\text{ for } k \in \mathbb{N}
\Big\}\,.
\end{equation}
 Replacing $\ell_{p}$ by the space $c_{0}$ of null sequences and proceeding in the same way $c_{0}(A)$ is defined.

\begin{remark} \label{proj}
Assume that $\big( X_k, i_{k}\big)_{k\in\mathbb{N}}$ and $\big( Y_k, j_{k}\big)_{k\in\mathbb{N}}$ are two  projective spectra of normed spaces, and denote by $\pi_{m}$ and $\rho_{m}$ (for each $m$) the corresponding projection into
$X_{m}$ and $Y_{m}$. If we have a family of bounded operators $\{ S_k: X_k \to Y_k \}_{k}$ satisfying $S_k \circ i_k = j_k \circ S_{k+1}$ for each $k$, then the operator
\[
S: \proj X_k \to \proj Y_k \, \text{ given by } \, (x_k) \mapsto (S_k x_k)
\]
is obviously well defined. It is also continuous, since for each $m$ we have $\rho_{m} \circ S = S_{m} \circ \pi_{m}$.  Clearly, if all $S_k$ are continuous bijections with continuous inverse, then $S$ is an isomorphism of pre-Fr\'echet spaces.
\end{remark}

\begin{remark}  \label{proj1}
If in a projective limit  $X =\proj X_k$ we consider the canonical seminorm defined in \eqref{clementi}, it is easily seen that $X_{p_n} =\bigoplus_{k=1}^n X_{k}$ holds isometrically. Taking  the cartesian product of finitely many precompact
(resp. nuclear) operators in Banach spaces  again leads to a precompact (resp. nuclear) operator. As a consequence, $X$  is Schwartz (resp. nuclear) whenever for each $k$ there is $m > k$ such that the canonical mapping from $X_m$ into $X_k$ given
 by $i_{m,k}=i_{m-1} \circ \cdots \circ i_k$ is precompact (resp. nuclear).
\end{remark}

\section{Banach space protagonists}
Here we recall some definitions and facts on the 'underlying Banach spaces' of the Fr\'echet spaces of
$\lambda$-Dirichlet series, which we later intend to study. In Section~\ref{sec.montel} we add
new information on Hardy spaces of general Dirichlet series, which seems of independent interest.

\subsection{Hardy spaces}
As we have already seen, $\mathcal{D}_{\infty}(\lambda)$ is perhaps the most important space within the theory of general Dirichlet series; but it is not the only one.  There is also the scale of Hardy spaces $\mathcal{H}_{p}(\lambda)$ of Dirichlet
series, which was introduced in \cite{defant2019hardy}.\\
Given a frequency, a $\lambda$-Dirichlet polynomial is just a finite $\lambda$-Dirichlet series  $\sum_{n=1}^{N} a_{n} e^{-\lambda_{n}s}$. For every such polynomial and $1 \leq p < \infty$
\begin{equation*}
\bigg(\lim_{T \to \infty} \frac{1}{2T} \int_{-T}^{T}\Big|\sum^{N}_{n=1} a_{n}e^{-\lambda_{n}it} \Big|^{p} dt\bigg)^{\frac{1}{p}}\,
\end{equation*}
exists, and in this way one defines a norm on the space of all $\lambda$-Dirichlet polynomial. The Hardy space $\mathcal{H}_{p} (\lambda)$ is defined as the completion of this space. \\
This definition makes the space difficult to handle. There is however a different, more convenient approach that links these spaces with Fourier analysis on groups (see \cite[Section~3]{defantschoolmann2019Hptheory}). This requires a little bit of  preparation.\\

First of all, all characters on the group $(\mathbb{R}, +)$ are of the form $t\mapsto e^{-ixt}$, where $x\in \mathbb{R}$. Now, let $G$ be a compact abelian group and
$\beta : (\mathbb{R}, +) \to G$ a continuous homomorphism with dense range. Then for every character $\gamma \in \hat{G}$ there is a (unique) $x \in \mathbb{R}$ so that $\gamma \circ \beta(t) = e^{-itx}$ for all $t\in \mathbb{R}$ (for simplicity we write $\gamma=h_{x}$). We then
identify $\hat{G} = \{ h_{x} \colon x \in \hat{\beta} (\hat{G}) \}$. With this, the pair $(G, \beta)$ is said to be a $\lambda$-Dirichlet group if for every $n \in \mathbb{N}$ there is a (unique) character $\gamma \in \hat{G}$ so that $\gamma = h_{\lambda_{n}}$.\\
For every frequency such an object exists. The Bohr compactification  $\overline{\mathbb{R}}:=\widehat{(\mathbb{R},d)}$ of $\mathbb{R}$ with $d$ the discrete topology together with the embedding
$
\beta_{\overline{\mathbb{R}}}\colon \mathbb{R} \hookrightarrow \overline{\mathbb{R}}$ given by $x \mapsto \left[ t \mapsto e^{-ixt}\right]\,,
$
forms a Dirichlet group,
which obviously for any arbitrary frequency $\lambda$ serves as a $\lambda$-Dirichlet group.
Below we indicate that special $\lambda$s often allow $\lambda$-Dirichlet groups which are more adjusted
to the concrete structure of the sequence.\\

Given a $\lambda$-Dirichlet group $(G, \beta)$  and $1 \leq p \leq \infty$, the Hardy space $H_{p}^{\lambda} (G)$ is defined as the closed subspace of $L_{p}(G)$ consisting of those $f$ whose Fourier
coefficients
\[
\widehat{f} (h_{x}) = \int_{G} f(t) \overline{h_{x}(t)} d \mu(t)
\]
are $0$ whenever $x \not\in \{\lambda_{n} \colon n \in \mathbb{N}  \}$. With this the space $\mathcal{H}_{p}(\lambda)$ is defined as
\[
\mathcal{H}_{p}(\lambda) =\big\{ \sum{a_{n}e^{-\lambda_{n} s}} \colon \text{ there is (a unique)\,}  f \in H_{p}^{\lambda}(G),  \text{ with } \; a_{n}=\widehat{f}(h_{\lambda_{n}})  \text{ for all }  n\in\mathbb{N}\big\}\,,
\]
and the definition does not depend on the choice of the $\lambda$-Dirichlet group
\cite[Theorem~3.24]{defant2019hardy}.
This is a Banach space with the norm given by $\big\Vert \sum{a_{n}e^{-\lambda_{n} s}} \big\Vert_{\mathcal{H}_{p}(\lambda)}=\Vert f \Vert_{L_{p}(G)}$,
whenever $\sum{a_{n}e^{-\lambda_{n} s}}$ and $f$ are related to each other. Let us note that for $1 \leq p < \infty$ this Banach space coincides with the definition that we gave above (see \cite[Theorem 3.26]{defantschoolmann2019Hptheory}) -- but, moreover, in this way we have a proper definition for $\mathcal{H}_{\infty}(\lambda)$.\\

We finish this section by describing $\lambda$-Dirichlet groups for some of our basic examples of frequencies, and what do the corresponding Hardy spaces look like. For  $\lambda=(\log n)$ we denote by $\mathfrak{p}=(p_{n})$
the sequence of prime numbers. Then the infinite dimensional torus  $\mathbb{T}^{\infty}:=\prod_{n=1}^{\infty} \mathbb{T}$ (with its natural group structure) together with the so-called Kronecker flow
\[
\beta_{\mathbb{T}^{\infty}}\colon \mathbb{R} \to \mathbb{T}^{\infty} \text{ defined as } t \mapsto \mathfrak{p}^{-it}=(2^{-it},3^{-it}, 5^{-it}, \ldots),
\]
gives a  $(\log n)$-Dirichlet group. Then $f \in H_p^{(\log n)}(\mathbb{T}^\infty)$ if and only if $f \in L_p(\mathbb{T}^\infty)$ and the Fourier coefficient $\hat{f}(\alpha) = 0$ for any finite sequence
$\alpha = (\alpha_k)$ of integers with  $\alpha_k < 0$ for some $k$. In other terms,
\[
H_p(\mathbb{T}^\infty): = H_p^{(\log n)}(\mathbb{T}^\infty) = \mathcal{H}_p((\log n))
\]
holds isometrically, and $h_{\log n} = z^\alpha$ whenever $n=\mathfrak{p}^\alpha$.
So the above definition of $\mathcal{H}_{p}(\lambda)$  actually coincides with Bayart's definition of
$\mathcal{H}_{p}((\log n))$ in the ordinary case given in  \cite{bayart2002hardy}.\\
The second example is the frequency  $\lambda=(n)=(0,1,2,\ldots)$. Then $G:=\mathbb{T}$ together with $\beta_{\mathbb{T}}(t):=e^{-it}$ is a $(n)$-Dirichlet group, and $\mathcal{H}_{p}((n))$ equals the classical Hardy space $H_{p}(\mathbb{T}):= H_{p}^{(n)}(\mathbb{T}) $.

\subsection{Almost periodic functions} \label{periodic}

A continuous function $g : \mathbb{R} \to \mathbb{C}$ is said to be uniformly almost periodic if for every $\varepsilon >0$ there is $\ell >0$ so
that for every interval $I \subseteq \mathbb{R}$ of length $\ell$ there exists $\tau \in I$ such that
\[
\sup_{x \in \mathbb{R}} \vert g(x) - g(x + \tau) \vert < \varepsilon \,.
\]
Equivalently, a continuous function $g\colon \mathbb{R} \to \mathbb{C}$ is uniformly almost periodic if and only if  it is the uniform limit of trigonometric polynomials of the form $\sum_{n=1}^{N} a_{x_{n}} e^{-itx_{n}}$, where $x_{n} \in \mathbb{R}$ (see \cite[Chapter~1, \S5, $2^\circ$~Theorem, p. 29]{Be54}).\\
A function $f : [ \re > \sigma_{0}] \to \mathbb{C}$ is said to be  uniformly almost periodic if for every for $\sigma > \sigma_{0}$ the function $\mathbb{R} \to \mathbb{C}$ defined as $t \mapsto f(\sigma + i t)$ (which we also will sometimes denote by $f_\sigma=f (\sigma + i \punkt)$) is uniformly  almost periodic.
Given $f : [ \re > \sigma_{0}] \to \mathbb{C}$, bounded and uniformly almost periodic,  for each $x \in \mathbb{R}$ the corresponding Bohr coefficient is defined as
\begin{equation} \label{Bohrcoeffholo}
a_{x}(f) = \lim_{T \to \infty} \frac{1}{2T} \int_{-T}^{T} f(\sigma + it) e^{(\sigma + it)x} dt \,,
\end{equation}
where the integral is convergent for every $\sigma > \sigma_{0}$ and
independent of each such $\sigma$ (see
\cite[page~157]{Be54}). These coefficients are $0$ except for at most countably many $x$, and $f=0$ if and only if $a_{x}(f)=0$ for every
$x$ (see  \cite[pages~158 and~18]{Be54}). The reader is referred to \cite{Be54} for more details on almost periodic functions. \\

With all this, for a given frequency $\lambda$, the space $\HRe$ is defined in  \cite[Definition~2.15]{defant2020riesz} as consisting of those functions $f : [\re >0] \to \mathbb{C}$ which are bounded, holomorphic,
and uniformly almost periodic such that $a_{x}(f)=0$ unless $x = \lambda_{n}$ for some $n$.\\

The following result from \cite[Theorem~2.16]{defant2020riesz} characterizes the limit functions of
Dirichlet series in $\mathcal{H}_\infty(\lambda)$ in terms of almost periodicity.

\begin{theorem} \label{equivalence2}
For every frequency $\lambda$ the identification $f \mapsto
\sum a_{\lambda_n}(f) e^{-\lambda_n s}$ defines  an isometric bijection
\[
\HRe = \mathcal{H}_\infty(\lambda)
\]
preserving Bohr and Dirichlet coefficients.
\end{theorem}

\subsection{Montel theorems} \label{sec.montel}

Bayart showed in \cite[Lemma~18]{bayart2002hardy} that if $\big(\sum a_{n}^{N} n^{-s} \big)_{N}$ is a bounded sequence in $\mathcal{D}_{\infty} ((\log n))$, then there is a subsequence $(N_{k})_{k}$ and a
Dirichlet series $\sum a_{n} n^{-s} \in \mathcal{D}_{\infty} ((\log n))$ so that $\big(\sum a_{n}^{N_{k}} n^{-s} \big)_{k}$ converges to $\sum a_{n} n^{-s}$ uniformly on $[\re > \sigma]$ for every $\sigma >0$. This result, often 
known as Bayart's Montel theorem for Dirichlet series, has become one of the cornerstones of the modern, functional-analytic approach  to ordinary Dirichlet series (see \cite[Theorem~3.11]{defant2018Dirichlet} or \cite[Theorem~6.3.1]{queffelec2013diophantine}).\\
Extending this to general Dirichlet series and to spaces other than $\mathcal{D}_{\infty}(\lambda)$ has been a major concern. It is known that such a result holds for $\mathcal{D}_{\infty}(\lambda)$ if and only if Bohr's theorem holds
for $\lambda$ (see Theorem~\ref{equivalence} below). As for Hardy spaces, \cite[Theorem~5.8]{defant2020variants} shows that an analogous result holds for $\mathcal{H}_{p}(\lambda)$ (with $1 \leq p \leq \infty$) if the frequency
satisfies Bohr's theorem. We show now that this assumption is actually not needed, and that such a Montel-type theorem in fact holds for every frequency. This is one of our main tools, and it seems of independent interest
for the structure theory of Hardy spaces of general Dirichlet series .

\begin{theorem} \label{teo montel3}
Let $\lambda=(\lambda_{n})$ be a frequency and $1\leq p \leq \infty$. For every  bounded sequence
$\big(\sum a_{n}^{(N)} e^{-\lambda_{n} s} \big)_N$ in $\mathcal{H}_{p}(\lambda)$, there exists a subsequence $(N_{k})_{k}$ and a $\lambda$-Dirichlet series
$\sum a_{n} e^{-\lambda_{n} s} \in \mathcal{H}_{p}(\lambda)$
so that
\[
\lim_{k\to \infty} ~\sum a_{n}^{(N_{k})} e^{-\lambda_{n} \sigma} e^{-\lambda_{n} s} =\sum a_{n}e^{-\lambda_{n} \sigma} e^{-\lambda_{n} s}
\]
in $\mathcal{H}_{p}(\lambda)$ for every $\sigma >0$.
\end{theorem}

As a matter of fact, we prove a more general result for  uniformly almost periodic functions (Theorem~\ref{teo montel}), from which this follows. We need some preliminary work. We begin by fixing some notation and collecting basic properties of the main tools that we are going to use. All of them are rather standard, and can be found in several monographs, like e.g. \cite{pereyraward,rudin_groups}. First of all, the Fourier transform
of a function $f \in L_{1} (\mathbb{R})$ is denoted either by $\mathcal{F}(f)$ or $\widehat{f}$ and is defined as
\[
\mathcal{F}(f) (t) = \widehat{f} (t) = \int_{-\infty}^{+\infty} f(t) e^{- i t x} dx  \,,
\]
for $t \in \mathbb{R}$.  The F\'ejer kernel is defined,  for $x >0$, as
\[
K_{x} (t) = \frac{1}{2\pi x} \Big( \frac{\sin ( x t/2)}{ t/2} \Big)^{2}
\]
for $t \in \mathbb{R}$. These belong to $L_{1}(\mathbb{R})$ and $\Vert K_{x} \Vert_{1} = 1$ for every $x$. The family $\{K_{x} \}_{x >0}$ is a summability kernel (i.e. $K_{x} * f \to f$ in $L_{1}(\mathbb{R})$ as $x \to \infty$ for every $f \in L_{1}(\mathbb{R})$). Also it is not difficult to check that
\begin{equation} \label{telemann}
\widehat{K}_{x} (t) = \Big( 1- \frac{\vert t \vert}{x} \Big) \chi_{[-x,x]} (t) \,,
\end{equation}
where $\chi_{A}$ denotes the indicator function of the set $A$.
The Poisson kernel is defined for $\sigma >0$ as
\[
P_{\sigma} (t) = \frac{1}{\pi}\frac{\sigma}{t^{2} + \sigma^{2}}
\]
for  $t \in \mathbb{R}$. Again, this belongs to $L_{1}(\mathbb{R})$ with $\Vert P_{\sigma} \Vert_{1}=1$, and
for every $\sigma,t$
\begin{equation} \label{morricone}
\widehat{P}_{\sigma} (t) = e^{- \vert t \vert \sigma} \,.
\end{equation}
Finally, given a frequency $\lambda$ and $D=\sum a_{n}e^{-\lambda_{n}s}\in \mathfrak{D}(\lambda)$, for each $x > 0$  the corresponding Riesz mean of $D$ of order $1$ (see \cite{defant2020riesz}) is given by
\[
R_{x}^{\lambda}(D)= \sum_{\lambda_{n}<x}a_n\big(1-\frac{\lambda_{n}}{x}\big) e^{-\lambda_n s} \,.
\]
Since every $f \in \HRe$ formally defines the Dirichlet series $D=\sum a_{\lambda_{n}}(f)e^{-\lambda_{n}s}$, the Riesz mean of $f$ of length $x$ and order $1$ is given by the entire function
\[
R_{x}^{\lambda}(f)(s) = \sum_{\lambda_{n}<x}a_{\lambda_{n}}(f)\big(1-\frac{\lambda_{n}}{x}\big) e^{-\lambda_n s} \,.
\]
Let us note that, for a given $\lambda$-Dirichlet series $D=\sum a_{n}e^{-\lambda_{n}s}$, the result \cite[Lemma~3.8]{schoolmann2018bohr} shows
\begin{equation}\label{pontes}
\inf \Big\{ \sigma \in \mathbb{R} \colon (R_{x}^{\lambda}(D))_{x}  \text{ converges uniformly on } [\re > \sigma] \Big\}\le \limsup_{x \to \infty} \frac{\log \big( \sup_{\re s > 0}\big\vert R_{x}(D)(s) \big\vert \big)}{x}\,.
\end{equation}
 With this we have at hand everything we need to proceed. We begin by isolating some observations.

\begin{lemma} \label{lemma0.1}
Let $\lambda=(\lambda_{n})$ be an arbitrary frequency, and $f  \in \HRe$.

\begin{enumerate}[(i)]
\item \label{lemma0.1-1} For all $\sigma,\varepsilon, x>0$ and $t\in \mathbb{R}$ we have
\begin{equation} \label{vector}
R_{x}^{\lambda}(f)(\sigma+\varepsilon+it)=(f_{\varepsilon}*P_{\sigma}*K_{x})(t) \,.
\end{equation}

\item \label{lemma0.1-2}  $\|R_{x}^{\lambda}(f)\|_{\infty}\le \|f\|_{\infty}$ for every $x>0$. In particular, $ (R^\lambda_{x}(f))_{x}$ converges uniformly to $f$ on all half-planes $[\re > \sigma]$, $\sigma>0$.

\item \label{lemma0.1-3} $\big(\sup_{\re s = \sigma}|f(s)|\big)_{\sigma >0}$ is decreasing in $\sigma >0$, and
$\|f\|_{\infty} = \lim_{\sigma \to 0} \sup_{\re s = \sigma}|f(s)|$\,.

\item \label{lemma0.1-4} For all $\sigma >0$ we have
\[
\sup_{\re s \geq \sigma} \vert f(s) \vert = \sup_{\re s = \sigma} \vert f(s) \vert \,.
\]
\end{enumerate}
\end{lemma}

\begin{proof}
\ref{lemma0.1-1} Let us take in first place a $\lambda$-polynomial $Q(t) = \sum_{n \in F} c_{n} e^{- \lambda_{n}i t}$, where $F$ is finite. Then, for fixed $\sigma >0$ and $ t \in \mathbb{R}$ we have, using \eqref{telemann} and \eqref{morricone}
\begin{equation} \label{berceuse}
\begin{split}
(Q * P_{\sigma} * K_{x}) (t)
=  \sum_{n \in F} c_{n} & e^{- \lambda_{n} i t} \mathcal{F} (P_{\sigma} * K_{x}) (\lambda_{n})
=   \sum_{n \in F} c_{n} e^{- \lambda_{n} i t} \widehat{P}_{\sigma} ( \lambda_{n}) \widehat{K}_{x} (\lambda_{n}) \\
& =   \sum_{n \in F} c_{n} e^{- \lambda_{n} i t}  e^{- \lambda_{n} \sigma} \Big(1-\frac{\lambda_{n}}{x} \Big) \chi_{[-x,x]} ( \lambda_{n})
=   \sum_{n \in F \atop \lambda_{n} < x} c_{n} e^{- \lambda_{n} (\sigma + i t)}  \Big(1-\frac{\lambda_{n}}{x} \Big) \,.
\end{split}
\end{equation}
Fix now $\varepsilon >0$. Since $f_{\varepsilon}$ is  uniformly almost periodic, there exists a sequence $(Q_{N}^{\varepsilon})_{N}$ of $\lambda$-polynomials that converge to $f_{\varepsilon}$ uniformly on $[\re >0]$. Moreover,
\[
\lim_{N\to \infty} a_{\lambda_{n}} (Q_{N}^{\varepsilon} ) = a_{\lambda_{n}}(f_{\varepsilon})=a_{\lambda_{n}}(f)e^{-\varepsilon \lambda_{n}}
\]
for every $n$.
Hence, applying \eqref{berceuse} (for $Q_{N}^{\varepsilon}$) and letting $N \to \infty$ yields the claim in \eqref{vector}.\\
In order to prove \ref{lemma0.1-2}, let us note that \eqref{vector} immediately implies
\[
\sup_{t\in \mathbb{R}} |R_{x}^{\lambda}(f)(\sigma+\varepsilon+it)|
\leq \|f_{\varepsilon}\|_{\infty} \|P_{\sigma}\|_{1} \|F_{x}\|_{1}\le \|f\|_{\infty} \,.
\]
Tending $\varepsilon, \sigma \to 0$ we obtain the claim. The uniform convergence on half-planes follows immediately from \eqref{pontes}, and the fact that the Bohr coefficients determine uniquely an almost periodic function (see \cite{Be54}).\\
Since the Riesz means are finite sums we have (see \cite[Section~2]{schoolmann2018bohr} or \cite[Lemma~1.7]{defant2018Dirichlet}, from which the argument can be adapted)
\[
\sup_{\re s = \sigma} \big\vert R_{x}^{\lambda} (f) (s) \big\vert
= \sup_{\re s \geq \sigma} \big\vert R_{x}^{\lambda} (f) (s) \big\vert \,.
\]
This gives
\[
\sup_{t \in \mathbb{R}} \big| R^\lambda_{x}(f)(\mu + it)  \big|
\leq
\sup_{t \in \mathbb{R}} \big| R^\lambda_{x}(f)(\sigma + it)  \big|\,,
\]
for $0 <\sigma < \mu$, and \ref{lemma0.1-3} follows from \ref{lemma0.1-2}. Once we have this, \ref{lemma0.1-4} follows immediately from the fact that $f(\sigma + \punkt)$ is uniformly almost periodic, that $\sup_{\re s = \rho} \vert f(s) \vert$ increases as $\rho \to \sigma^{+}$ and that $\sup_{\re s = \sigma} \vert f(s) \vert<~\infty$.
\end{proof}

We now proceed to the announced Montel-type theorem for $\HRe$.

\begin{theorem} \label{teo montel}
Let $\lambda$ be an arbitrary frequency. For every  bounded sequence $( f_{N} )_{N}$
in $\HRe$, there is a subsequence $( f_{N_k} )_{k}$ and $f  \in \HRe$ such that
$f_{N_k}(\sigma + \punkt) \to  f(\sigma +  \punkt)$ in $\HRe$ for every $\sigma >0$.
\end{theorem}

\begin{proof}
Since $|a_{\lambda_{n}}(f_{N})|\leq \|f_{N}\|_{\infty} \leq \sup_{N} \|f_{N}\|_{\infty}<\infty$ for all $n$ and $N$, a standard diagonal process provides us with a subsequence $(N_{k})_{k}$ so that
$\big( a_{\lambda_{n}}(f_{N_k}) \big)_{k}$ converges (in $\mathbb{C}$) for every $n$ as $k\to \infty$. Define
\[
a_{\lambda_{n}}:=\lim_{k\to \infty} a_{\lambda_{n}}(f_{N_k})
\]
for each $n$. Using Lemma~\ref{lemma0.1}--\ref{lemma0.1-2} we conclude that
\[
\Big| \sum_{\lambda_{n}<x}a_{\lambda_{n}} e^{-\lambda_n \sigma} \big( 1-\frac{\lambda_{n}}{x} \big) e^{-i\lambda_n t}\Big|
= \lim_{k\to \infty} |R_{x}^{\lambda}(f_{N_{k}})(\sigma+it)|
\leq \sup_N\|f_{N}\|_{\infty}<\infty
\,,
\]
for every $x >0$ and $\sigma >0$. Then \eqref{pontes} gives
\[
s \mapsto \sum_{\lambda_{n}<x}a_{\lambda_{n}}  \big( 1-\frac{\lambda_{n}}{x} \big) e^{-\lambda_n s}
\]
converges as $x\to \infty$ uniformly on every half plane $[\re > \sigma]$ with $\sigma >0$, and  let us denote the limit function by $f$. The uniform convergence on half planes easily gives that $f \in H_{\infty}^{\lambda}[\re >0]$ with Bohr coefficients $a_{\lambda_{n}}(f)=a_{\lambda_{n}}$, and it only remains to see that $\lim_{k\to \infty}f_{N_k}(\sigma +  \punkt)=f(\sigma+ \punkt)$ in $H_{\infty}^\lambda [\re  > 0]$ for every $\sigma >0$. Therefore fix $\sigma >0$ and observe that, by Lemma~\ref{lemma0.1}--\ref{lemma0.1-4}  it suffices to check that
\[
\lim_{k\to \infty}\sup_{t \in \mathbb{R}}|f_{N_k}(\sigma + i t)-  f(\sigma+it)| =0 \,.
\]
Let us note first that by  Lemma~\ref{lemma0.1}--\ref{lemma0.1-2} letting $x \to \infty$ in \eqref{vector} gives $g_{\sigma/2} * P_{\sigma/2}= g(\sigma + i \punkt)$  for every $g \in H_{\infty}^{\lambda}[\re >0]$. This, together with Lemma~\ref{lemma0.1}--\ref{lemma0.1-1} yields,
\begin{align*}
\sup_{t  \in \mathbb{R}}|R_{x}^{\lambda}(g)(\sigma + it) -g(\sigma + it)|
=
\sup_{t  \in \mathbb{R}} |g_{\sigma/2}*P_{\sigma/2}*F_{x}(t)-g_{\sigma/2}\ast & P_{\sigma/2}(t)| \\
& \leq \|g\|_{\infty} \|P_{\sigma/2}-P_{\sigma/2}*F_{x}\|_{L_{1}(\mathbb{R})} \,.
\end{align*}
for every such $g$ and $x >0$.
Being $(K_{x})_{x}$ a summability kernel, the latter term tends to $0$ as $x \to \infty$. Hence, given $\varepsilon >0$ we can find $x_{0}$ so that for all $x>x_{0}$
\begin{gather*}
\sup_{t \in \mathbb{R}} |f(\sigma+it) - R_{x}^{\lambda}(f) (\sigma+it)| \leq \frac{\varepsilon}{3}
 \\
\sup_{N\in \mathbb{N}} \sup_{t \in \mathbb{R}}| f_{N}(\sigma+it) - R_{x}^{\lambda}(f_{N})(\sigma+it)| \leq  \frac{\varepsilon}{3} \,.
\end{gather*}
On the other hand, since $\lim_{k\to \infty} a_{\lambda_{n}} (f_{N_{k}})= a_{\lambda_{n}}$ for every $n$, fixing $x =2x_{0}$ we may find $k_{0}$ so that
\[
\sum_{\lambda_{n}<x} |a_{\lambda_{n}}-a_{\lambda_{n}}(f_{N_k})| \leq \frac{\varepsilon}{3} \,,
\]
for every $k \geq k_{0}$. Joining all this together, given $k \geq k_{0}$ and $t \in \mathbb{R}$ we have
\begin{multline*}
\vert f(\sigma + it) - f_{N_{k}} (\sigma + it) \vert \\
\leq \vert f(\sigma + it) - R_{x}^{\lambda}(f) (\sigma+it) \vert + \vert R_{x}^{\lambda}(f) (\sigma+it) - R_{x}^{\lambda}(f_{N_{k}}) (\sigma+it)
\vert + \vert f(\sigma + it) -R_{x}^{\lambda}(f_{N_{k}}) (\sigma+it) \vert  \,,
\end{multline*}
and the previous three estimates complete the proof.
\end{proof}

We go now for a moment to vector-valued functions, considering for a given Banach space $X$, the space $\HRevv{X}$ defined in the obvious way. Also in this case $X$-valued almost periodic functions on $\mathbb{R}$, like in the scalar case, are
uniformly  approximable by $X$-valued almost periodic polynomials (see \cite[page~15]{amerioprouse}). Then the proof of Theorem~\ref{teo montel} can be followed step by step, replacing modulus by norms to get the following vector valued version.

\begin{theorem} \label{teo montel2}
Let $\lambda$ be an arbitrary frequency and $X$ a Banach space. Assume that  $( f_N )_{N}$
is bounded in $\HRevv{X}$, and that
there is a subsequence $(N_k)$ for which $\big( a_{\lambda_n}(f_{N_k}) \big)_{k}$ is convergent  for all $n$. Then there is $f  \in \HRevv{X}$ such that
$\lim_{k\to \infty}f_{N_k}(\sigma +  \punkt)= f(\sigma +  \punkt)$ in $\HRevv{X}$
for every $\sigma >0$.
\end{theorem}

As the final ingredient for the proof of Theorem~\ref{teo montel3} let us recall that by \cite[Lemma~4.9]{defantschoolmann2019Hptheory}, for each $1 \leq p < \infty$  there is an isometric embedding
\[
\Psi \colon \mathcal{H}_{p}(\lambda) \hookrightarrow
\HRevv{\mathcal{H}_{p}(\lambda)}
\]
so that, if $f = \Psi \big( \sum a_{n} e^{-\lambda_{n} s} \big)$, then $a_{\lambda_{n}}(f) = a_{n} e^{-\lambda_{n} z}  \in \mathcal{H}_{p}(\lambda)$.

\begin{proof}[Proof of Theorem~\ref{teo montel3}]
The case $p=\infty$ follows immediately from Theorems~\ref{teo montel} and~\ref{equivalence2}. The case $1 \leq p  < \infty$ is going to follow from Theorem~\ref{teo montel2}, combined with the action of the
embedding $\Psi$. To begin with, let us recall that $\vert a_{n}^{(N)} \vert \leq \sup_{N\in \mathbb{N}} \big\Vert \sum a_{n}^{(N)} e^{-\lambda_{n} s} \big\Vert_{ \mathcal{H}_{p} (\lambda)}=:C$. Then a
diagonal argument shows that we can find a subsequence $(N_{k})_{k}$ so that $\big(a_{n}^{N_{k}} \big)_{k}$ converges for every $n$. Let us define
\[
a_n:=\lim_{k\to \infty} a_n^{(N_k)} \,,
\]
and consider the (formal) $\lambda$-Dirichlet series $\sum a_{n} e^{-\lambda_{n}s}$. Our aim now is to check that this belongs to $\mathcal{H}_{p}(\lambda)$, and that it is the limit of the subsequence
of $\lambda$-Dirichlet series.\\
For each $N$ take the function $f_{N} = \Psi  \big( \sum a_{n}^{(N)} e^{-\lambda_{n} s} \big)$ and note that
\begin{equation*}
\lim_{k} a_{\lambda_{n}}(f_{N_{k}})
= \lim_{k} a_{n}^{(N_{k})} e^{-\lambda_{n} z}
= a_{n} e^{-\lambda_{n} z}
\end{equation*}
exists (in $\mathcal{H}_{p}(\lambda)$) for every $n$. Now we can use Theorem~\ref{teo montel2} to find some
$f \in \HRevv{\mathcal{H}_{p}(\lambda)}$ such that $\lim_{k\to \infty} f_{N_{k}} (\sigma + \punkt)=f(\sigma + \punkt)$ in $\HRevv{\mathcal{H}_{p}(\lambda)}$ for every $\sigma >0$ as $k\to \infty$.\\
Now, since $\Psi$ is isometric and $\Psi(\sum a_{n}^{N_{k}}e^{-\lambda_{n}\sigma}e^{-\lambda_{n}s})=f_{N_{k}}(\sigma+\punkt)$, the sequence $\sum a_{n}^{N_{k}}e^{-\lambda_{n}\sigma}e^{-\lambda_{n}s}$ is Cauchy in $\mathcal{H}_{p}(\lambda)$ and so converges with limit $\sum a_{n}e^{-\lambda_{n}\sigma} e^{-\lambda_{n}s}$. Hence $\sum a_{n}e^{-\lambda_{n}\sigma}e^{-\lambda_{n}s}\in \mathcal{H}_{p}(\lambda)$ for every $\sigma>0$ with $\|\sum a_{n}e^{-\lambda_{n}\sigma}e^{-\lambda_{n}s}\|_{p}\le C$, and so by \cite[Theorem~4.7]{defantschoolmann2019Hptheory} indeed $\sum a_{n}e^{-\lambda_{n}s}\in \mathcal{H}_{p}(\lambda)$. Moreover, the
fact that the embedding $\Psi$ is isometric gives
\[
\Big\Vert \sum a_{n}^{(N_{k})} e^{-\lambda_{n} \sigma} e^{-\lambda_{n} s}- \sum a_{n}e^{-\lambda_{n} \sigma} e^{-\lambda_{n} s} \Big\Vert_{\mathcal{H}_{p}(\lambda)}
 = \| f(\sigma + i \punkt)- f^{N_k}(\sigma + i \punkt)\|_{\infty} \,,
\]
for every $\sigma >0$. This completes the proof.
\end{proof}

Given a somewhere convergent $\lambda$-Dirichlet series $D= \sum a_{n} e^{- \lambda_{n} s}$ and $\sigma >0$ we define the translates series as $D_{\sigma} = \sum a_{n} e^{- \lambda_{n} \sigma} e^{- \lambda_{n} s}$. Note that if the
first series converges at $s$, then so also does the translated series and  $D_{\sigma}(s) = D(s +\sigma)$ (justifying the name). It is easy to see that the translation operator $\tau_{\sigma} \colon \mathcal{H}_{p}(\lambda)\to \mathcal{H}_{p}(\lambda)$ given by
\begin{equation} \label{def:translation}
 \sum a_{n}e^{-\lambda_{n}s} \mapsto \sum a_{n}e^{-\lambda_{n}\sigma}e^{-\lambda_{n}s}
\end{equation}
is well defined and continuous. As a straightforward consequence of Theorem~\ref{teo montel3} we can say more.

\begin{corollary} \label{uhu}
For every $\sigma>0$ and $1\le p \le \infty$ the translation operator $\tau_{\sigma} \colon \mathcal{H}_{p}(\lambda)\to \mathcal{H}_{p}(\lambda)$ is compact.
\end{corollary}

We recall the following notions going back to  Bohr (see \cite{defantschoolmann2019Hptheory}). An infinite matrix $R=(r^{n}_{k})_{n,k \in \mathbb{N}}$ of rational numbers is called Bohr matrix
whenever each row  $R_n = (r_k^n)_k$ is finite, i.e. $r^{n}_{k} \neq 0$ for only finitely many $k$s.
Given  a sequence $\lambda=(\lambda_{n})$ of real numbers,  a (finite or infinite) sequence $B=(b_{k})$ in $\mathbb{R}$ is said to be a basis for $\lambda$ if it is  $\mathbb{Q}$-linearly independent and for  each $n$ there is a finite sequence  $(r^{n}_{k})_{k}$ of  rational coefficients
such that $\lambda_{n}=\sum_k r^{n}_{k} b_{k}$. In this case, te matrix  $R=(r^{n}_{j})_{n,j}$ is said to be a  Bohr matrix of $\lambda$ with respect to the basis $B$. If $\lambda$ is a frequency, such a basis always exists
(in fact it  can be chosen as a subsequence of $\lambda$), and if $R$ is the associated  Bohr matrix $R$, we write  $\lambda=(R,B)$. Observe that neither $B$ nor  $R$  need to be unique.\\

With this, the $N$th Abschnitt (for $N \in \mathbb{N}$) of a $\lambda$-Dirichlet series $D = \sum a_{n} e^{- \lambda_{n}s}$ is the series
\[
D\restrict{N}  = \sum_{\lambda_{n} \in \spanned_{\mathbb{Q}} \{b_{1} , \ldots , b_{N}  \} }
a_{n} e^{- \lambda_{n}s}\,.
\]
To illustrate this let us note that, for ordinary Dirichlet series (i.e. $\lambda = (\log n)$), the $N$th Abschnitt of a Dirichlet series is built by taking the coefficients $a_{n}$ for which $n$ depends only on the first $N$ primes. It
is well known (see e.g. \cite[Corollary~13.9]{defant2018Dirichlet}) that a Dirichlet series belongs to $\mathcal{H}_{p}((\log n))$ if and only if its $N$th Abschnitt belongs to  $\mathcal{H}_{p}((\log n))$ for every $N$, and their norms (in $\mathcal{H}_{p}((\log n))$)
are bounded.
It was shown in  \cite[Theorem~5.9]{defant2020variants} that an analogous result holds for $\mathcal{H}_{p}(\lambda)$ whenever $\lambda$ satisfies Bohr's theorem. This is an immediate consequence of
\cite[Theorem~5.8]{defant2020variants}, a Montel-type theorem for frequencies satisfying Bohr's theorem. Now that we have dropped this hypothesis in Theorem~\ref{teo montel3}, we can proceed exactly as in \cite[Theorem~5.9]{defant2020variants}, to have the following version for arbitrary frequencies.

\begin{corollary}\label{abschnitt}
Let $\lambda$ be a frequency with a decomposition $(B,R)$, and $1 \leq p \leq \infty$. A $\lambda$-Dirichlet series $D$ belongs to $\mathcal{H}_{p}(\lambda)$ if and only if $D\restrict{N} \in \mathcal{H}_{p}(\lambda)$ for all $N$ and $\sup_{N} \Vert D\restrict{N} \Vert_{\mathcal{H}_{p}(\lambda)} < \infty$. Moreover, in this case, $\Vert D \Vert_{ \mathcal{H}_{p} (\lambda)} = \sup_{N} \Vert D\restrict{N} \Vert_{\mathcal{H}_{p}(\lambda)}$.
\end{corollary}

\subsection{Equivalence theorem} \label{sec:equivalence}

As we have repeatidly mentioned, the space $\mathcal{D}_{\infty}(\lambda)$ is one of the main actors within the theory of general Dirichlet series. It consists of all $\sum a_n e^{-\lambda_n s}$ which converge on $[\re >0]$ such that the limit function $f(s)=\sum_{n=1}^{\infty}a_{n}e^{-\lambda_{n}s}: [\re >0] \to \mathbb{C}$ is bounded. Together with
\[
\Big\| \sum a_n e^{-\lambda_n s}  \Big\|_{\infty}=\sup_{s\in [\re >0]} |f(s)|
\]
we obtain a normed space (see \cite{schoolmann2018bohr} or \cite{defant2019hardy}). Since  the limit function $f$ of every $\sum a_n e^{-\lambda_n s} \in \mathcal{D}_{\infty}(\lambda)$ belongs to $\HRe$ (see e.g. \cite[Corollary~3.9]{schoolmann2018bohr}), where $a_{n}=a_{\lambda_{n}}(f)$ for all $n$, we may identify $\mathcal{D}_{\infty}(\lambda)$ with the subspace of all $f\in \HRe$, which are represented by their Dirichlet series, that is we have $f(s)=\sum_{n=1}^{\infty} a_{\lambda_{n}}(f)e^{-\lambda_{n}s}$ for every $s\in [\re >0]$.
 In particular, for every $\sum a_{n}e^{-\lambda_{n}s}\in \mathcal{D}_{\infty}(\lambda)$ have that
 \[
 a_n = \lim_{T \to \infty} \frac{1}{2T}\int_{-T}^{T} f( \sigma + it) e^{(\sigma + it) \lambda_n} dt\,,
 \]
for all  $n\in \mathbb{N}$ and $\sigma>0$, which implies
 \begin{equation} \label{boundednesscoeff}
  \sup_{n \in \mathbb{N}} |a_{n}| \le \Big\| \sum a_n e^{-\lambda_n s}  \Big\|_{\infty} \,.
 \end{equation}
But in general $\left(\mathcal{D}_{\infty}(\lambda), \|\punkt\|_{\infty} \right)$ is no Banach space, or equivalently it does not  form a closed subspace of $\HRe$ (see \cite[Theorem~5.2]{schoolmann2018bohr}). The normed space  $\mathcal{D}_{\infty}(\lambda)$, and in particular the question when it is complete, was extensively studied in \cite{CaDeMaSc_VV,defant2019hardy,defant2020variants,schoolmann2018bohr, schoolmann2018bohrA}.\\

One of the celebrated results in the theory of ordinary Dirichlet series is due to  Hedenmalm, Lindqvist and Seip \cite{HLS} and  shows that
\[
 \mathcal{D}_\infty((\log n)) = \mathcal{H}_\infty((\log n))
\]
isometrically and coefficient preserving. The result reflects  that the theory of ordinary Dirichlet series generating  bounded, holomorphic functions on the positive half plane, is intimately linked with Fourier analysis on the group
$\mathbb{T}^{\infty}$, and its proof uses  Diophantine approximation just to mention one of its crucial tools. The question of whether or not an analogous equality holds for arbitrary frequencies has also called a deal of attention over the last years.
Another important topic is Bayart's version of Montel theorem (see Section~\ref{sec.montel}). The question here was to find out for which frequencies does $\mathcal{D}_{\infty}(\lambda)$ satisfy Montel's theorem: every  bounded
sequence  $\big( \sum a_n^Ne^{-\lambda_{n}s} \big)_{N}$ of Dirichlet series in $\mathcal{D}_{\infty}(\lambda)$ admits a subsequence $(N_{k})$ and $\sum a_n e^{-\lambda_{n}s} \in \mathcal{D}_{\infty}(\lambda)$ such that $\big( \sum a_n^{N_{k}} e^{-\lambda_{n}s} \big)_{k}$ converges to $\sum a_n e^{-\lambda_{n}s}$ uniformly on
$[\re >\sigma]$ for every $\sigma>0$ as $k\to \infty$ (or, to put it in other terms, Theorem~\ref{teo montel} holds entirely for the subspace $\mathcal{D}_{\infty}(\lambda)$ of $\HRe$).\\

All these questions were clarified in  \cite[Theorem~5.1]{defant2020variants}, showing that actually they are all equivalent to each other, and equivalent to  $\lambda$ satisfying Bohr's theorem. We recall here the result.

\begin{theorem} \label{equivalence}
For every frequency $\lambda$ the following statements are equivalent:
\begin{enumerate}[(i)]
\item Bohr's theorem holds for $\lambda$,
\item  $\mathcal{D}_{\infty}(\lambda)$ is a Banach space,
\item   $\mathcal{D}_{\infty}(\lambda)$ satisfies  Montel's theorem,
\item  $\mathcal{D}_\infty(\lambda) = \mathcal{H}_\infty(\lambda)$, isometrically and coefficient preserving,
\item  $\mathcal{D}_\infty(\lambda) = \HRe$ , isometrically and coefficient preserving.
\end{enumerate}
\end{theorem}

Therefore, under the three  concrete conditions given in Section~\ref{diri} (in particular for our main examples $(\log n)^{\alpha}$, $(n)$ and $(\log p_{n})$) all these statements are equivalent.

\section{Pre-Fr\'echet spaces generated by abscissas} \label{Dseries}

Following an idea from \cite{andreaspabloantonio}, we suggest an abstract approach to define certain (pre-)Fr\'echet spaces of $\lambda$-Dirichlet series derived from some pre-existing normed space. In a first step, given a  normed spaces
of $\lambda$-Dirichlet series $\mathfrak{X}(\lambda)$  (satisfying certain conditions that we explicit later), we define
\begin{itemize}
\item the  abscissa $\sigma_{\mathfrak{X}(\lambda)}(D)$ associated to  $\mathfrak{X}(\lambda)$ for each $\lambda$-Dirichlet series $D$,
\end{itemize}
and then  in a second step generate the
 \begin{itemize}
\item  space $\mathfrak{X}_+(\lambda)$ of all $\lambda$-Dirichlet series  for which $\sigma_{\mathfrak{X}(\lambda)}(D)\leq 0$ (which is, as we will see,  pre-Fr\'echet)\,.
\end{itemize}
We will later apply this general procedure to study the spaces $\mathcal{D}_{\infty,+}(\lambda)$  and   $\mathcal{H}_{p,+}(\lambda)$ for $1 \leq p \leq \infty$, generated by  $\mathcal{D}_{\infty}(\lambda)$ and the
$\mathcal{H}_{p}(\lambda)$s. We will also use an analogous procedure to define the space $\HRep$ (that consists of uniformly almost periodic functions) from $\HRe$.

\subsection{Abscissas} \label{Abscissassection}
Given a frequency $\lambda$, we  consider normed spaces $\mathfrak{X}(\lambda)$  of $\lambda$-Dirichlet series satisfying the following three requirements:
\begin{enumerate}[label=(AS\arabic*)]
\item \label{as1}
All monomials $e^{-\lambda_n s} $ belong to $\mathfrak{X}(\lambda)$ and have norm $1$.
In particular, all $\lambda$-Dirichlet polynomials $\sum_{n=1}^N a_n e^{-\lambda_n s}$ belong to $\mathfrak{X}(\lambda)$, and
\[
 \Big\| \sum_{n=1}^N a_n e^{-\lambda_n s}  \Big\|_{\mathfrak{X}(\lambda)} \leq \sum_{n=1}^N |a_n|\,.
\]

\item \label{as2}
All coefficient functionals $\mathfrak{X}(\lambda) \to \mathbb{C}$ given by $\sum a_{n} e^{- \lambda_{n} s} \mapsto a_n$, are  uniformly bounded.
In particular,  there is some $C \ge 0$ such that for all $\lambda$-Dirichlet polynomials $\sum_{n=1}^N a_n e^{-\lambda_n s}$ we have
\[
\max_{1 \leq n \leq N}|a_n| \leq C \Big\|  \sum_{n=1}^N a_n e^{-\lambda_n s}  \Big\|_{\mathfrak{X}(\lambda)} \,.
\]

\item \label{as3}
For every $\sigma > 0$ the translation operator
\[
 \tau_\sigma: \mathfrak{X}(\lambda) \to \mathfrak{X}(\lambda)\,,
\]
defined as in \eqref{def:translation} is well defined and bounded.
\end{enumerate}
Whenever this is the case, we say that the space $\mathfrak{X}(\lambda)$ is $\lambda$-admissible. We also define the subspace
\[
\mathfrak{X}^0(\lambda) = \Big\{   \sum a_n e^{-\lambda_n s}\in \mathfrak{X}(\lambda)
\colon \forall \sigma > 0 ,\, \Big( \sum_{\lambda_{n}<x} a_n e^{-\lambda_n \sigma}e^{-\lambda_n s}  \Big)_{x}
\, \text{ converges in } \mathfrak{X}(\lambda)\, \Big\}\,.
\]
This is again a $\lambda$-admissible space. Note that, if the sequence of monomials $\{ e^{-\lambda_{n} s} \}_{n}$ constitutes a basis of $\mathfrak{X}(\lambda)$, then $\mathfrak{X}(\lambda) =\mathfrak{X}^0(\lambda)$. We show now some examples of admissible spaces.

\begin{example} \label{restes}
Let $\lambda$ be any frequency.
\begin{enumerate}[(a)]
\item \label{restes-A}
Let us fix some Banach space  $X$  of complex sequences satisfying the following two properties: (1)
the $e_k$s  form a normalised basis of $X$ (so in particular, $\ell_1  \subset X \subset c_0$), and (2) if $(a_n) \in X$, then
$(e^{-\lambda_n \sigma} a_n) \in X$.\\
Examples of such $X$ are $\ell_p$ for $1 \leq p < \infty$ and $c_0$. Another relevant Banach space for our purposes is $\Sigma$, defined as the  linear  space  of all  complex sequences $(a_n)$ such that $\sum a_n$ converges,
normed by $\|(a_n)\|_{\Sigma} = \sup_N |\sum_{n=1}^N a_n |$.
Property $(1)$ is straightforward, and $(2)$ may be either proved directly or by observing that, for a given $(a_n) \in \Sigma$,
the Dirichlet series $\sum a_n e^{-\lambda_n s}$ converges in $s = 0$, hence also on $[\re >0]$.
Let us observe that  $\Sigma$ is isomorphic to $c_0$, with the identification  given by $ (a_{n})  \mapsto \big(\sum_{n=N}^{\infty}a_{n} \big)_{N}$.\\

We define $\mathcal{D}_X(\lambda)$ as the linear space of all $\lambda$-Dirichlet series $\sum a_n e^{-\lambda_n s}$ such that $(a_n) \in X$. Together with the norm $\big\Vert \sum a_{n} e^{-\lambda_{n} s} \big\Vert_{\mathcal{D}_X(\lambda)}
= \|(a_n)\|_X$ it is easy to see that  $\mathcal{D}_{X}(\lambda)$ is a $\lambda$-admissible Banach space. Since we here (by definition)  identify $\mathcal{D}_X(\lambda)$ and $X$ as Banach spaces, the  monomials
$\{e^{-\lambda_n s}\}_n$ form a basis of $\mathcal{D}_X(\lambda)$ and, in particular, $\mathcal{D}_X(\lambda) = \mathcal{D}^0_X(\lambda)$.\\

Especially interesting for us are the  $\lambda$-admissible Banach spaces $ \mathcal{D}_{\ell_p}(\lambda)$, $ \mathcal{D}_{\Sigma}(\lambda)$ and $\mathcal{D}_{c_0}(\lambda)$. Moreover, for any
$X$ as above we clearly have
\[
  \mathcal{D}_{\ell_1}(\lambda)
  \subset
  \mathcal{D}_{X}(\lambda)
  \subset
  \mathcal{D}_{c_0}(\lambda)\,.
\]

\item For each $1 \leq p \leq \infty$ the space $\mathcal{H}_p(\lambda)$ is  $\lambda$-admissible. Moreover, $\mathcal{H}_p(\lambda)=\mathcal{H}^0_p(\lambda)$ for $1 < p <~\infty$, since in this case the sequence $\{e^{-\lambda_n s}\}_n$ forms a basis (see \cite[Theorem~4.16]{defant2019hardy}).
Note that $\mathcal{D}_{\ell_2}(\lambda) = \mathcal{H}_2(\lambda)$ as Banach spaces (identifying $(a_{n})_{n}$ with $\sum a_{n} e^{-\lambda_{n}s}$).

\item The space  $\mathcal{D}_\infty(\lambda)$ is clearly a $\lambda$-admissible normed space, and if $\lambda$ satisfies  Bohr's theorem (recall  Section~\ref{diri}) it coincides with  $\mathcal{D}^0_\infty(\lambda)$.
Moreover, we have from Theorem~\ref{equivalence} that $\mathcal{D}_\infty(\lambda) = \mathcal{H}_\infty(\lambda)$ iff $\lambda$ satisfies  Bohr's theorem.
\end{enumerate}
\end{example}

Given a $\lambda$-admissible space $\mathfrak{X}(\lambda)$ the $\mathfrak{X}(\lambda)$-abscissa of an arbitrary  $\lambda$-Dirichlet series $D=\sum a_{n}e^{-\lambda_{n}s}$ is defined as
\begin{equation}\label{carapaz}
\sigma_{\mathfrak{X}(\lambda)}(D)  = \inf  \Big\{ \sigma \in \mathbb{R} \colon \sum a_{n} e^{-\lambda_n \sigma}
\, e^{-\lambda_n s} \in \mathfrak{X}(\lambda)   \Big\}  \, \in  \, [-\infty, \infty]\,.
\end{equation}
Also, the $\mathfrak{X}^{0}(\lambda)$-abscissa (recall that this is also a $\lambda$-admissible space) of $D=\sum a_{n}e^{-\lambda_{n}s}$ is the of infimum all real $\sigma$ for which the partial sums $\Big(\sum_{1}^{N} a_{n}e^{-\lambda_{n}\sigma}e^{-\lambda_{n}s}\Big)_{N}$ converges in $\mathfrak{X}(\lambda)$.

\begin{example}\label{exA}
As in \cite{andreaspabloantonio}, the classical abscissas of convergence  (recall Section~\ref{diri}) can be reformulated in terms of abscissas of certain admissible spaces.
Let $\lambda$ be a frequency, and $D \in \mathfrak{D}(\lambda)$. Then
\begin{enumerate}[(a)]
\item $\sigma_c(D) = \sigma_{\mathcal{D}_\Sigma(\lambda)}(D)$
\item $\sigma_a(D) = \sigma_{\mathcal{D}_{\ell_1}(\lambda)}(D)$
\item $\sigma_b(D) = \sigma_{\mathcal{D}_\infty(\lambda)}(D)$
\item \label{exA4} $\sigma_u(D) = \sigma_{\mathfrak{X}^0(\lambda)}(D)$, where $\mathfrak{X}(\lambda)=\mathcal{D}_{\infty}(\lambda)$ or $\mathcal{H}_{\infty}(\lambda)$.
\end{enumerate}
\end{example}

A useful tool  for the understanding of such abscissas are the so-called Bohr-Cahen formulas
for $\sigma_i(D)$ with $i=c,u,a$. A careful analysis of the typical proofs shows how to extend these formulas to our abstract setting
(see e.g \cite{defant2018Dirichlet,queffelec2013diophantine} and in particular \cite[Proposition~2.2]{andreaspabloantonio}), provided that $\mathfrak{X}(\lambda)$ is a Banach space.

\begin{proposition}\label{Prop:BasicEstimationsAbscissae}
Let $\mathfrak{X}(\lambda)$ be a $\lambda$-admissible Banach space of $\lambda$-Dirichlet series. Then for every $D = \sum a_{n} e^{-\lambda_n s} \in \mathfrak{D}(\lambda)$ we have
\[
\sigma_{\mathfrak{X}^{0}(\lambda)}(D) \leq  \limsup_{x \to \infty}
\frac{\log \big\| \sum_{\lambda_{n}<x} a_{n} e^{-\lambda_n s} \big\|_{\mathfrak{X}(\lambda)}}{x}  \,,
\]
where equality holds whenever the abscissa is non-negative. 
%
\end{proposition}

\subsection{The space} \label{baronnoir}

Given a $\lambda$-admissible normed space $\mathfrak{X}(\lambda)$, we define the vector space
\begin{equation} \label{fx}
\mathfrak{X}_+(\lambda) = \big\{ D \in \mathfrak{D}(\lambda)
\colon
\sigma_{\mathfrak{X}(\lambda)}(D) \leq 0   \big\}\,,
\end{equation}
which consists of all $\lambda$-Dirichlet series so that every translation belongs to $\mathfrak{X}(\lambda)$. Our first task now is to endow this space with some structure. To begin with,
for each $k \in \mathbb{N}$ the expression
\begin{equation}\label{step}
\|D\|_{\mathfrak{X}(\lambda),k} = \|D_{1/k}\|_{\mathfrak{X}(\lambda)}
\end{equation}
defines a norm, so that the sequence  $\big( \|\punkt\|_{\mathfrak{X}(\lambda),k} \big)_{k}$ endows $\mathfrak{X}(\lambda)$ with a pre-Fr\'echet topology. Our second step is to give a representation
as a projective limit of a countable projective spectrum of normed spaces. In fact, we do it in two different ways. \\
On the one hand, for each $k$ we consider the space
\begin{equation} \label{hammerklavier}
\mathfrak{X}_k(\lambda) = \big\{ D \in \mathfrak{D}(\lambda)
\colon
  D_{1/k} \in \mathfrak{X}(\lambda) \big\} \,,
\end{equation}
on which $\|\punkt\|_{\mathfrak{X}(\lambda),k}$ defines a norm. If $i_k: \mathfrak{X}_{k+1}(\lambda) \hookrightarrow \mathfrak{X}_{k}(\lambda) $ is the canonical
injection, then the pair
\begin{equation} \label{spec1}
\big(\mathfrak{X}_k(\lambda), i_k  \big)_{k \in \mathbb{N}}\,
\end{equation}
forms a countable projective spectrum of normed spaces.\\
On the other hand, for each $k$ we define the mapping $\tau_k: \mathfrak{X}(\lambda) \hookrightarrow \mathfrak{X}(\lambda)$ by $D \mapsto  D_{\frac{1}{k}- \frac{1}{k+1}}$. Again, the pair
\begin{equation} \label{spec2}
\big(\mathfrak{X}(\lambda), \tau_k  \big)_{k \in \mathbb{N}}\,,
\end{equation}
defines a  countable projective spectrum of normed spaces.\\
Let us observe that
\begin{equation} \label{cruc}
\varphi_k: \mathfrak{X}_k(\lambda) \to \mathfrak{X}(\lambda) \, \text{ given by } \,
\sum a_n e^{-\lambda_n s}\to \sum a_n
e^{-\frac{\lambda_n}{k} }e^{-\lambda_n s}
\end{equation}
is an isometric bijection, where the inverse is given by $\varphi_{k}^{-1}(\sum a_{n}e^{-\lambda_{n}s})=\sum a_{n}e^{\frac{\lambda_{n}}{k}}e^{-\lambda_{n}s}$. With this,
it is plain that the spectra defined in \eqref{spec1} and \eqref{spec2} are equivalent, in the sense that
\begin{equation} \label{trompa}
\tau_k \circ\varphi_{k+1} = \varphi_{k} \circ i_k
\end{equation}
for all $k$ (see Remark~\ref{proj}). This leads to the two announced  representations of the pre-Fr\'echet space $\mathfrak{X}_+(\lambda)$ as a projective limit of a countable spectrum of normed spaces.

\begin{proposition} \label{projlimitX}
Let $\lambda$ be a frequency and $\mathfrak{X}(\lambda)$ be a $\lambda$-admissible normed space. Then  $\mathfrak{X}_+(\lambda)$ is a pre-Fr\'echet space, which is a Fr\'echet space whenever
$\mathfrak{X}(\lambda)$ is a Banach space.
Also, the mappings
\[
\mathfrak{X}_+(\lambda) =  \proj (\mathfrak{X}_{k}(\lambda), i_k) \,
\text{given by } \,
D \mapsto (D)_{k=1}^\infty
\]
and
\[
\mathfrak{X}_+(\lambda) =  \proj (\mathfrak{X}(\lambda), \tau_k)\,
\text{given by } \,
D \mapsto (D_{1/k})_{k=1}^\infty
\]
are isomorphisms of pre-Fr\'echet spaces.
\end{proposition}
\begin{proof}
By \eqref{trompa} (see Remark~\ref{proj}) it is enough to check this just for the first representation. We denote the mapping by $\Phi$, which is clearly linear and injective. By the very definition of the
spectrum, if $(D^{k}) \in \proj (\mathfrak{X}_{k}(\lambda), i_k)$, then there is some $D \in \mathfrak{D}(\lambda)$ so that $D^{(k)} = D$ for every $k$. Note that this implies that
$\sigma_{\mathfrak{X}(\lambda)}(D) \leq 1/k$ for all $k$, hence $D \in \mathfrak{X}_{+}(\lambda)$ and clearly $\Phi (D) = (D^{(k)})$, so that $\Phi$ is surjective. Finally, if $\pi_{k}: \proj (\mathfrak{X}_{k}(\lambda), i_k) \to \mathfrak{X}_{k}(\lambda)$
is the canonical projection, one easily gets that both $\pi_{k} \circ \Phi$ and $\Phi^{-1} \circ \pi_{k}^{-1}$ are continuous for every $k$; hence both $\Phi$ and $\Phi^{-1}$ are continuous. This completes the
argument. If $\mathfrak{X}(\lambda)$ is complete, then the description as a  projective limit yields the completeness of $\mathfrak{X}_{+}(\lambda)$.
\end{proof}

\begin{remark} \label{gould}
Given a frequency  $\lambda$, we define the K\"othe matrix
\begin{equation} \label{rusalka}
A(\lambda) = (e^{-\frac{\lambda_n}{k}})_{n,k=1}^\infty\,.
\end{equation}
As a consequence of Proposition~\ref{projlimitX} and Remark~\ref{proj} we have
\begin{equation} \label{kharris}
\mathcal{D}_{\ell_p,+}(\lambda) = \ell_{p}(A(\lambda)) \, \text{ and }  \mathcal{D}_{c_0,+}(\lambda) = c_0(A(\lambda))\,  \,,
\end{equation}
where in both cases $\sum a_{n} e^{-\lambda_{n} s}$ is identified with the sequence $(a_{n})_{n}$.
Note  the particular case $\mathcal{H}_{2,+}(\lambda) = \ell_{2}(A)$.
All these spaces are Fr\'echet-Schwartz, since diagonal operators on $\ell_p$ or
$c_0$ are compact whenever the diagonal is a zero sequence (see again Remark~\ref{proj1}).
We show now, in a more general context, that
\begin{equation} \label{no}
\mathcal{D}_{\Sigma,+}(\lambda)\hookrightarrow \mathcal{D}_{c_0,+}(\lambda) = c_0(A(\lambda))\,.
\end{equation}
(as sequence spaces), but that in contrast to \eqref{kharris} this inclusion  in general is strict.\\

If $\mathfrak{X}(\lambda)$ is an admissible normed space, then by \ref{as2}, we can find some $C\geq 1$ so that
\[
\vert a_{n} e^{-\lambda_{n}/k} \vert \leq C \big\Vert \textstyle \sum a_{n} e^{-\lambda_{n}s} \Vert_{\mathfrak{X}(\lambda),k}
\]
for every  $\sum a_{n} e^{-\lambda_{n}s} \in \mathfrak{X}_{+}(\lambda)$ and all $k$. Then, given some $k$ we can pick any $k < m$ to have
\[
\vert a_{n} e^{-\lambda_{n}/k} \vert \leq C \big\Vert \textstyle \sum a_{n} e^{-\lambda_{n}s} \Vert_{\mathfrak{X}(\lambda),m} \vert  e^{-\lambda_{n} (1/k - 1/m)} \vert
\]
This shows that $(a_{n})_{n} \in c_{0}(A(\lambda))$ and the inclusion
\[
\mathfrak{X}_+(\lambda)\hookrightarrow c_{0}(A(\lambda))
\]
is continuous. Taking $\mathfrak{X}_+(\lambda) =\mathcal{D}_{\Sigma,+}(\lambda)$ this gives the inclusion in \eqref{no}. Note that, for $\lambda= (\log n)$, the series $\zeta = \sum n^{s}  \notin
\mathcal{D}_{\Sigma,+}(\lambda)$ (because  $\sigma_{\mathcal{D}_{\Sigma}(\lambda)}(\zeta)= \sigma_c(\zeta) = 1$, recall Example~\ref{exA}), but the sequence of coefficients belongs to  $c_{0}(A (\log n))$. This shows that the inclusion is (as we announced) in general strict.\\
On the other hand, if $\sum a_{n}e^{-\lambda_{n}s}\in \mathcal{D}_{\ell_1,+}(\lambda)$, then by \ref{as1}, the sequence $(\sum_{n=1}^{N} a_{n}e^{-\lambda_{n}/k}e^{-\lambda_{n}s})_{N}$   is Cauchy for every $k$.
If $\mathfrak{X}(\lambda)$ is complete, then the sequence converges and, by \ref{as2}, it does it to $\sum a_{n}e^{-\lambda_{n}/k}e^{-\lambda_{n}s}$, that therefore belongs to $\mathfrak{X}(\lambda)$.
This shows that $\sum a_{n}e^{-\lambda_{n}s}\in \mathfrak{X}^{0}_{+}(\lambda)$ or, to put it in other terms
\begin{equation} \label{mistake}
\mathcal{D}_{\ell_1,+}(\lambda) \hookrightarrow \mathfrak{X}^{0}_{+}(\lambda)
\end{equation}
(with continuous inclusion) for every $\lambda$-admissible Banach space $\mathfrak{X}(\lambda)$. This in particular gives, for every $\lambda$-admissible Banach space $\mathfrak{X}(\lambda)$,  the canonical continuous
inclusions
\begin{equation} \label{passt}
\lambda_1(A (\lambda)) = \mathcal{D}_{\ell_1,+}(\lambda) \, \hookrightarrow \, \mathfrak{X}^{0}_{+}(\lambda)
\,\hookrightarrow \,\mathfrak{X}_{+}(\lambda) \,\hookrightarrow \,
 \mathcal{D}_{c_0,+}(\lambda)
 =
c_0(A(\lambda))\,.
\end{equation}
This implies that, in this case, $\sigma_{\mathfrak{X}(\lambda)}(D) \leq \sigma_a(D)$ for every $D \in \mathfrak{D}(\lambda)$. In view of
this, and since for every $D \in \mathfrak{D}(\lambda)$ we have the well know inequalities $\sigma_{c}(D)\le \sigma_{b}(D) \le \sigma_{u}(D)\le \sigma_{a}(D)$,  one may wonder if $\sigma_{c}(D)\le \sigma_{\mathfrak{X}(\lambda)}(D)$ for every $\lambda$-admissible Banach space $\mathfrak{X}(\lambda)$ and $D \in \mathfrak{D}(\lambda)$. But this is false -- take $D =\sum \frac{1}{n^{1/2}} n^{-s}$,
then $\sigma_{c}(D)= 1/2$, but $\sigma_{\mathcal{H}_{2,+}}(D)= 0$.
\end{remark}

\subsection{Bases}

\begin{proposition} \label{basisA}
Let $\mathfrak{X}(\lambda)$ be a $\lambda$-admissible space. Then the following are equivalent:
\begin{enumerate}[(i)]
\item \label{basisA1}
The sequence of monomials $e^{-\lambda_n s}$ forms a basis of  $\mathfrak{X}_+(\lambda)$,
\item \label{basisA2}
$\mathfrak{X}(\lambda) = \mathfrak{X}^0(\lambda)$,
\item \label{basisA3}
$\sigma_{\mathfrak{X}(\lambda)}(D)=\sigma_{\mathfrak{X}^0(\lambda)}(D)$ for all $D\in\mathcal{D}(\lambda)$.
\end{enumerate}
In particular, if the monomials $\{e^{-\lambda_n s}\}_n$ are a basis of $\mathfrak{X}(\lambda)$, then they  also form a basis
of $\mathfrak{X}_+(\lambda)$.
\end{proposition}
\begin{proof}
Assume that \ref{basisA1} holds and take $\sum a_{n} e^{-\lambda_{n}s} \in \mathfrak{X} (\lambda) \subset \mathfrak{X}_+(\lambda)$. Since the monomials form a basis of the latter, the
partial sums converge to the series for every seminorm $\Vert \punkt \Vert_{\mathfrak{X}(\lambda),k}$. That is to say that the partial sums of $\sum a_{n} e^{-\lambda_{n}/k} e^{-\lambda_{n}s}$ converge (in $\mathfrak{X}(\lambda)$) to the series. This
shows \ref{basisA2}. Clearly \ref{basisA2} implies \ref{basisA3}. To finish the proof suppose that \ref{basisA3}  holds, and let us show \ref{basisA1}. Take  $\sum a_{n} e^{-\lambda_{n}s} \in \mathfrak{X}_{+}(\lambda)$.
By assumption $\sigma_{\mathfrak{X}^0(\lambda)}(D)=\sigma_{\mathfrak{X}(\lambda)}(D) \leq 0$ and therefore, given any $k \in \mathbb{N}$, we have
\[
\sum a_n e^{-\lambda_n /k} e^{-\lambda_n s} \in \mathfrak{X}(\lambda)\,,
\]
and $\Big( \sum_{n=1}^{N} a_n e^{-\lambda_n /k  + \sigma} e^{-\lambda_n s} \Big)_{N}$  converges for every $\sigma >0$ (and, by \ref{as2} it has to do it to the series itself). Since this holds for every $k$ we immediately have that \[
\lim_N \sum_{n=1}^N a_n e^{-\lambda_n /k} e^{-\lambda_n s} = \sum a_n e^{-\lambda_n /k} e^{-\lambda_n s}
\]
in  $\mathfrak{X}(\lambda)$, which by \eqref{cruc} implies
\[
\lim_N \sum_{n=1}^N a_n e^{-\lambda_n s} = \sum a_n  e^{-\lambda_n s}
\]
in  $\mathfrak{X}_k(\lambda)$ for every $k$. This yields the  conclusion.
\end{proof}

%
%
%
%

\subsection{Nuclearity} \label{sec:nuc}

We finish this section by figuring out when our spaces are nuclear (recall the definition in Section~\ref{prelim}). The Grothendieck-Pietsch theorem
\cite[Theorem~28.15]{meise1997introduction} is here our main tool:
a Fr\'echet space $E$ with a basis $\{e_{n} \}$ and  an increasing system of seminorms $\|\punkt\|_{k}$ is nuclear if and only if for every $k\in \mathbb{N}$ there is $m\in \mathbb{N}$ such that
\begin{equation} \label{GP}
\sum_{n=1}^{\infty} \|e_{n}\|_{k}\|e_{n}\|_{m}^{-1}<\infty \,.
\end{equation}
For the monomials \ref{as1} gives $\big\Vert e^{-\lambda_{n} s} \Vert_{\mathfrak{X}(\lambda),k} = e^{- \frac{\lambda_{n}}{k}}$, and then \eqref{GP} can be very conveniently reformulated.

\begin{lemma} \label{zimmermann}
Let $\lambda$ be any frequency. Then $L(\lambda)=0$ if and only if for every $k \in \mathbb{N}$ there exists $m>k$ so that
\begin{equation} \label{reger}
\sum_{n=1}^{\infty} e^{- \lambda_{n} (\frac{1}{k} - \frac{1}{m})} < \infty \,.
\end{equation}
\end{lemma}
\begin{proof}
Let us assume first that $L(\lambda)=0$. Given any $k \in \mathbb{N}$ just pick some $m > k$ and define $\varepsilon = \frac{1}{2} (\frac{1}{k} - \frac{1}{m})$. Since $L(\lambda)=0$ we can find $n_{\varepsilon}$ so that $\frac{\log n}{\lambda_{n}} < \varepsilon$ for every $n\geq n_{\varepsilon}$. Then
\[
\sum_{n \geq n_{\varepsilon}} e^{- \lambda_{n} (\frac{1}{k} - \frac{1}{m})}
= \sum_{n \geq n_{\varepsilon}} e^{- 2 \lambda_{n} \varepsilon}
\leq \sum_{n \geq n_{\varepsilon}} \frac{1}{n^{2}} \,,
\]
which clearly yields \eqref{reger}.
Conversely, given any $k$ take $m>k$ so that \eqref{reger} holds. This implies
\[
L(\lambda) = \sigma_{c} \big( \sum a_{n} e^{-\lambda_{n}s} \big) \leq \frac{1}{k} - \frac{1}{m} < \frac{1}{k} \,.
\]
Since $k$ was arbitrary, this gives $L(\lambda)=0$ and completes the proof.
\end{proof}

With this we can now say quite a bit about when are the spaces nuclear.

\begin{proposition} \label{nuclear2}
Let $\lambda$ be any frequency. Then the following are equivalent.
\begin{enumerate}[(i)]
\item  \label{nuclear22} $L(\lambda) =0$

\item \label{nuclear21} There is a $\lambda$-admissible normed space $X(\lambda)$ such that
 $\mathfrak{X}_{+}(\lambda)$ is a nuclear Fr\'echet space and  the monomials $\{e^{-\lambda_{n} s}\}_{n}$ are a basis.

\item \label{nuclear23} For every  $\lambda$-admissible Banach space  $X(\lambda)$ we have that
 $\mathfrak{X}_{+}(\lambda)$ is a nuclear Fr\'echet space and  the monomials $\{e^{-\lambda_{n} s}\}_{n}$ are a basis.
\end{enumerate}
Moreover, in this case all Fr\'echet spaces $\mathfrak{X}_{+}(\lambda)$ for all possible
$\lambda$-admissible Banach spaces  $\mathfrak{X}(\lambda)$ coincide, and so in particular
$
\mathfrak{X}_{+}(\lambda) =\ell_{p} (A(\lambda)) = c_{0}(A(\lambda))
$
for every $1 \leq p < \infty$, where $A(\lambda)$ is the K\"othe matrix defined in \eqref{rusalka}.
\end{proposition}
\begin{proof}
As  a straightforward consequence of the Grothedieck-Pietsch theorem \eqref{GP} and Lemma~\ref{zimmermann} we have that \ref{nuclear21} implies \ref{nuclear22}. Suppose now that
$L(\lambda)=0$ and choose any $\lambda$-admissible Banach space  $\mathfrak{X}(\lambda)$. From \eqref{passt} we have
\[
\ell_{1} (A) = \mathcal{D}_{1,+} (\lambda) \hookrightarrow \mathfrak{X}_{+}^{0}(\lambda)
\hookrightarrow \mathfrak{X}_{+}(\lambda) \hookrightarrow c_{0}(A(\lambda)) \,.
\]
Then Lemma~\ref{zimmermann} and \cite[Theorem~28.16]{meise1997introduction} (a consequence of the Grothendieck-Pietsch that characterises nuclearity in K\"othe spaces) imply that $\ell_{1} (A(\lambda)) =c_{0}(A(\lambda))$, the canonical vectors $e_{n} = (\delta_{n,j})_{j}$ are a basis and the spaces are nuclear. But, then,
\[
\ell_{1} (A(\lambda)) = \mathfrak{X}_{+}(\lambda) = c_{0}(A(\lambda)) \,,
\]
and the conclusion follows. Since the remaining implication is obvious, the proof is completed.
\end{proof}

\section{Fr\'echet space protagonists}  \label{frechet}

We now apply the abstract approach devised in the previous section to some concrete spaces of general Dirichlet series. This yields our main results.

\subsection{Bounded Dirichlet series}  \label{structure}

In the introduction we already defined $\mathcal{D}_{\infty,+} (\lambda)$ as  the space of all $\lambda$-Dirichlet series $\sum a_{n}e^{-\lambda_{n}s}$ that on $[\re >0]$ converge to
(a  necessarily holomorphic) function which is bounded  on all smaller planes $[\re > \sigma]$. Looking at \eqref{fx}, this is precisely the pre-Fr\'echet space generated by the $\lambda$-admissible space
$\mathcal{D}_\infty(\lambda)$. Obviously $\mathcal{D}_{\infty}(\lambda)$ is a linear subspace of $\mathcal{D}_{\infty,+} (\lambda)$. To see that that both spaces in general are different, note that
the ordinary Dirichlet series $\sum (-1)^{n} n^{-s}=(1-2^{-s})\zeta(s)$ generated by the Riemann zeta-series $\zeta = \sum n^{-s}$,   belongs to  $\mathcal{D}_{\infty,+}((\log n))$, but not to
$\mathcal{D}_\infty((\log n))$.\\
Note first that by~\eqref{boundednesscoeff} for every $k$ we have $\sup_{n} |a_{n}e^{- \frac{\lambda_{n}}{k}}| \leq \big\| \sum a_{n} e^{-\lambda_{n} s} \big\|_{k}$. This
in particular shows that the coordinate functionals $\sum a_{n} e^{-\lambda_{n} s} \mapsto a_{N}$
 are equicontinuous on $\mathcal{D}_{\infty,+} (\lambda)$.
Let us note that in this case the space defined in \eqref{hammerklavier} is exactly
\begin{equation*}
\mathcal{D}_{\infty,k} (\lambda) := \big\{ \sum a_{n} e^{- \lambda_{n} s} \in \mathfrak{D}(\lambda)
\colon \sum a_{n} e^{- \frac{ \lambda_{n}}{k}} e^{- \lambda_{n} s} \in \mathcal{D}_{\infty} (\lambda)\big\} \,,
\end{equation*}
endowed  with the norm
\[
\big\Vert \sum a_{n} e^{- \lambda_{n}s} \big\Vert_{\mathcal{D}_{\infty}(\lambda),k}
: = \big\Vert \sum a_{n} e^{- \frac{ \lambda_{n}}{k}}  e^{- \lambda_{n}s} \big\Vert_{\mathcal{D}_{\infty} (\lambda)}\,.
\]
Proposition~\ref{projlimitX} provides us with two different representations of the space $\mathcal{D}_{\infty,+} (\lambda)$ as a projective limit (recall the definitions in \eqref{spec1} and \eqref{spec2}).

\begin{proposition} \label{projlimit}
Let $\lambda$ be a frequency. Then  $\mathcal{D}_{\infty,+}(\lambda)$ is a pre-Fr\'echet Schwartz space which admits the following representations as projective limit
\[
\mathcal{D}_{\infty,+}(\lambda) =  \proj(\mathcal{D}_{\infty,k}(\lambda), i_k) %
=  \proj(\mathcal{D}_{\infty}(\lambda), \tau_k) \,.
\]
\end{proposition}
\begin{proof}
The projective descriptions are immediate from Proposition~\ref{projlimitX}. In order to see that $\mathcal{D}_{\infty,+}(\lambda)$ is Schwartz let us note that, by Corollary~\ref{uhu}, the
mappings $\tau_{k} : \mathcal{H}_{\infty}(\lambda) \to \mathcal{H}_{\infty}(\lambda)$ are compact for every $k$. Since  by Theorem~\ref{equivalence2} and \cite[Corollary~3.9]{schoolmann2018bohr} we know that
$\mathcal{D}_{\infty}(\lambda)$ is an isometric subspace of $\mathcal{H}_{\infty}(\lambda)$, we immediately deduce all mappings $\tau_{k} : \mathcal{D}_{\infty}(\lambda) \to \mathcal{D}_{\infty}(\lambda)$ are
compact and, then, $\proj(\mathcal{D}_{\infty}(\lambda), \tau_k)$ is Schwartz.
\end{proof}

We illustrate all this with  an  interesting example. Let $\lambda$ be a $\mathbb{Q}$-linearly independent frequency, and consider the K\"othe matrix $A(\lambda)$ defined in \eqref{rusalka}.
From \cite[Theorem~4.7]{schoolmann2018bohr} we know that
\[
\mathcal{D}_{\infty,k}(\lambda) \to \ell_1\big((e^{-\frac{\lambda_n}{k}})_{n}\big)\, \text{ given by } \sum a_n n^{-s} \mapsto (a_n)
\]
is an isometric isomorphism for every $k$. This immediately gives (see \eqref{Koethe})  that by making the indentification $ \sum a_n n^{-s} \mapsto (a_n)$ we have
\[
\mathcal{D}_{\infty,+} (\lambda)  = \ell_{1}(A (\lambda))
\]
as Fr\'echet spaces. \\

Our aim in the following sections is to study the structure of $\mathcal{D}_{\infty,+} (\lambda)$, which in the end will yield a sort of analogue of  Theorem~\ref{equivalence} (the 'equivalence theorem') for Fr\'echet spaces of general Dirichlet series.

\subsubsection{Completeness}

Let us recall that the proof of Theorem~\ref{equivalence} (see \cite[Lemma~5.2]{defant2020variants}) requires an application of the uniform boundedness
principle. Barrelled spaces is the biggest class of spaces on which the uniform boundedness principle holds. So, when moving to the framework of locally convex spaces, barrelledness appears as a natural
property in our setting. Our next result shows that this is indeed the case, and that it gives another equivalent reformulation of Bohr's theorem for $\lambda$. Note that this property is in some sense
hidden in the Banach case, since a normed space is barreled if and only if it is complete.

\begin{theorem} \label{equi-1}
For every frequency $\lambda$ the following statements are equivalent
\begin{enumerate}[(i)]
\item \label{equi-11} Bohr's theorem holds for $\lambda$.
\item \label{equi-12} $\mathcal{D}_{\infty,+}(\lambda)$ is a Fr\'{e}chet space.
\item \label{equi-15} $\mathcal{D}_{\infty,+}(\lambda)$ is barreled.
\end{enumerate}
\end{theorem}

\begin{remark}\label{dufay}
Before we proceed to the proof of the theorem let us point out that, if  for every $\sigma>0$ there is a constant $C=C(\sigma)$ such that
for every choice of finitely many  $a_1, \ldots, a_M \in \mathbb{C}$, we have
\begin{equation}\label{point}
\Big\|\sum_{n=1}^{N} a_{n}e^{-\sigma \lambda_{n}}e^{-\lambda_{n}s}\Big\|_{\mathcal{D}_{\infty}(\lambda)}
\le C \Big\|\sum_{n=1}^{M} a_{n}e^{-\lambda_{n}s}\Big\|_{\mathcal{D}_{\infty}(\lambda)} \,,
\end{equation}
for all $N \leq M$, then (following the argument in \cite[Theorem~4.12]{CaDeMaSc_VV}) Bohr's theorem holds for $\lambda$. Indeed, if \eqref{point} holds, we may take $D = \sum a_{n} e^{-\lambda_{n} s} \in \mathcal{D}_{\infty}^{\ext}(\lambda)$ and fix $N$. Then (see e.g. \cite[Proposition~3.4]{schoolmann2018bohr})
\[
\Big\|\sum_{n=1}^{N} a_{n}(1-\frac{\lambda_{n}}{x})e^{-\sigma \lambda_{n}}e^{-\lambda_{n}s}\Big\|_{\mathcal{D}_{\infty}(\lambda)}
\le C \|R_{x}^{\lambda}(D)\|_{\infty}
\le C_{1} \big\| \sum a_{n} e^{-\lambda_{n} s} \big\|_{\mathcal{D}_{\infty}(\lambda)},
\]
for every $x>N$. Now, letting $x\to \infty$ we get
\[
\Big\|\sum_{n=1}^{N} a_{n}e^{-\sigma \lambda_{n}}e^{-\lambda_{n}s}\Big\|_{\mathcal{D}_{\infty}(\lambda)}
\le C_{1}(\sigma)\big\| \sum a_{n} e^{-\lambda_{n} s} \big\|_{\mathcal{D}_{\infty}(\lambda)},
\]
which implies $\sigma_{u}(D)\le 0$ (use Proposition~\ref{Prop:BasicEstimationsAbscissae} with e.g. $\mathfrak{X}(\lambda)=\mathcal{D}_{\infty} (\lambda)$, see also Example \ref{exA}--\ref{exA4}); i.e. Bohr's theorem holds for $\lambda$.
\end{remark}

\begin{proof}[Proof of Theorem~\ref{equi-1}]\text{}

\noindent \ref{equi-11} $\Rightarrow$ \ref{equi-12} By Proposition~\ref{projlimit} we know that
$\mathcal{D}_{\infty,+}(\lambda) =  \proj(\mathcal{D}_{\infty}(\lambda), \tau_k)$ as pre-Fr\'echet spaces. From Theorem~\ref{equivalence} we know that $\mathcal{D}_{\infty} (\lambda)$ is complete,
so that  the latter projective limit is complete. The conclusion then follows.\\

\noindent  \ref{equi-12} $\Rightarrow$  \ref{equi-15} This follows from the general fact that every Fr\'echet space is barreled.\\

\noindent  \ref{equi-15} $\Rightarrow$  \ref{equi-11} As we have already shown, ir suffices to check that \eqref{point} holds. To do that, for each fixed $k$ we consider the family of operators $T_N: \mathcal{D}_{\infty,+}(\lambda) \to \mathbb{C}$ given by
\[
\sum a_{n} e^{-\lambda_{n} s} \mapsto \sum_{n=1}^N a_n e^{-\frac{\lambda_n}{k}}
\]
for $N \in \mathbb{N}$. Then $\{ T_{N}  \}_{N}$ is a  bounded set in the topological dual of $\mathcal{D}_{\infty,+}(\lambda)$, which (since the space is barrelled) is then equicontinuous.
In other terms, there is a constant $C = C(k) >0$ and $\ell > k$ such that for all $\sum a_{n} e^{-\lambda_{n}s} \in \mathcal{D}_{\infty,+}(\lambda)$ we have
\[
\sup_N \Big| \sum_{k=1}^N a_n e^{-\frac{\lambda_n}{k}}  \Big| \leq
C \big\| \sum a_{n} e^{-\lambda_{n}s}  \big\|_{\mathcal{D}_{\infty}(\lambda),\ell}\,.
\]
Finally, given $a_1, \ldots, a_M \in \mathbb{C}$  and $\re z >0$, we apply this to the series
$\sum_{n=1}^M a_n e^{-\lambda_n z} e^{-\lambda_n s}$, to have
\[
\sup_{\re z >0}  \sup_N \Big| \sum_{n=1}^N a_n e^{-\lambda_n z} e^{-\frac{\lambda_n}{k}}   \Big|
\leq C \Big| \sum_{n=1}^N a_ne^{-\frac{\lambda_n}{\ell}} e^{-\lambda_n z}  \Big|\,,
\]
and this gives ~\eqref{point}.
\end{proof}

\subsubsection{Montel}

The appearance of Montel's theorem for $\mathcal{D}_{\infty}(\lambda)$ in Theorem~\ref{equivalence} leads to another well known class of locally convex spaces: Montel spaces.
\begin{theorem} \label{equi-1a}
For every frequency $\lambda$ the following statements are equivalent
\begin{enumerate}[(i)]
\item \label{equi-11A} Bohr's theorem holds for $\lambda$.
\item \label{equi-12B} $\mathcal{D}_{\infty,+}(\lambda)$ is a Montel space.
\end{enumerate}
\end{theorem}

\begin{proof}
If \ref{equi-11A} holds, then  $\mathcal{D}_{\infty,+}(\lambda)$ by Proposition~\ref{projlimit} and Theorem~\ref{equi-1} is a Fr\'echet-Schwartz space, and these are always Montel spaces. Conversely,
since Montel spaces by their definition are barreled, Theorem~\ref{equi-1} also proves that \ref{equi-12B} implies \ref{equi-11A}.
\end{proof}

\subsubsection{Bases}

We know that the sequence of monomials   $\{ e^{-\lambda_{n}s} \}_{n}$ forms a basis of $\mathcal{D}_{\infty,+}((\log n))$ (see \cite[Theorem~2.2]{bonet2018frechet}) and for
$\mathcal{D}_{\infty,+}((n))$ (in this case $e^{-ns}$ corresponds to the monomial
$z^{n}$, and the result is classical). The following equivalence extends these.

\begin{theorem} \label{equi-1b}
For every frequency $\lambda$ the following statements are equivalent:
\begin{enumerate}[(i)]
\item \label{equi-11B} Bohr's theorem holds for $\lambda$.
\item \label{equi-16}
For every $k$ there are $\ell > k$ and $C>0$ such that for each $\sum a_n e^{-\lambda_n s} \in \mathcal{D}_{\infty, +}(\lambda)$ we have
\[
\sup_N \Big\| \sum_{n=1}^N a_n e^{-\lambda_n s}  \Big\|_{\mathcal{D}_{\infty} (\lambda),k} \leq C \big\| \sum a_n e^{-\lambda_n s} \big\|_{\mathcal{D}_{\infty} (\lambda),\ell}\,.
\]
\end{enumerate}
Moreover, in this case the monomials $\{ e^{-\lambda_{n}s} \}_n$ form a basis of $\mathcal{D}_{\infty,+}(\lambda)$.
\end{theorem}
Provided $\mathcal{D}_{\infty,+}(\lambda)$ is complete (or by Theorem~\ref{equi-1} equivalently Bohr's theorem holds for $\lambda$), observe that statement \ref{equi-16} is an immediate
consequence of \eqref{basisineq}, whenever $\{ e^{-\lambda_{n}s} \}$  forms a basis  (compare also with \cite[Theorem~14.3.6]{jarchow2012locally} or \cite[Lemma~28.10]{meise1997introduction}).
Unfortunately, in general having a basis for a pre-Fr\'echet  space does not necessarily imply the corresponding inequality \eqref{basisineq}, so that it would be interesting to
find a concrete frequency $\lambda$ (not satisfying Bohr's theorem) such that the sequence of monomials $\{e^{-\lambda_{n}s}\}$ forms a basis for $\mathcal{D}_{\infty,+}(\lambda)$,
but for which statement \ref{equi-16} fails.

\begin{proof}
Suppose that Bohr's theorem holds for $\lambda$ and, for each $N$ consider the operator $T_{N}\colon \mathcal{D}_{\infty,+}(\lambda)\to \mathcal{D}_{\infty,k}(\lambda)$ given by
\[
\sum a_{n} e^{-\lambda_{n}s} \mapsto \sum_{n=1}^{N} a_{n} e^{-\lambda_{n}s} \,.
\]
Each of these is bounded, and, since Bohr's theorem holds, the pointwise limit exits. Now, $\mathcal{D}_{\infty,+}(\lambda)$ is barreled (recall Theorem~\ref{equi-1}) and this gives that the family $(T_{N})_{N}$ is equicontinuous. This implies \ref{equi-16}.\\
Conversely, if \ref{equi-16} holds, this clearly implies \eqref{point} which, as we have seen, gives that Bohr's theorem holds for $\lambda$.\\
Finally note that if $\lambda$ satisfies Bohr's theorem, then $\mathcal{D}^0_{\infty}(\lambda) = \mathcal{D}_{\infty}(\lambda)$, and by Proposition~\ref{basisA} the monomials form a basis of $\mathcal{D}_{\infty,+}(\lambda)$.
%
%
%
%
%
%
%
\end{proof}

Section~\ref{diri} provides us with new examples of frequencies for which the monomials are a basis of $\mathcal{D}_{\infty,+}(\lambda)$. For instance, $\lambda_n =(\log n)^\alpha$
with $\alpha >0$, which satisfies Landau's condition, and so Bohr's theorem. So far we do not know what happens for the frequency $\lambda_n =\log \log n$.

\subsubsection{Nuclearity} \label{tour}

We face now the last property we are interested in: nuclearity.  Let us recall that by \cite{bonet2018frechet}, the space $\mathcal{D}_{\infty} ((\log n))$
is not nuclear, whereas $\mathcal{D}_{\infty} ((n))$ equals the space $H(\mathbb{D})$, which is well known to be nuclear
(see e.g. \cite[Corollary~8, page~499]{jarchow2012locally}). So the question arises naturally: for which frequencies are our spaces nuclear? Proposition~\ref{nuclear2} gives us the answer.

\begin{theorem} \label{nuclear}
Let $\lambda$ be any frequency. Then  $\mathcal{D}_{\infty,+}(\lambda)$ is a nuclear Fr\'echet space if and only if $L(\lambda) =0$.
\end{theorem}
\begin{proof}
If $\mathcal{D}_{\infty,+}(\lambda)$ is a Fr\'echet space, then by Theorems~\ref{equi-1} and~\ref{equi-1b} the monomials are a basis. So, if the space is also nuclear, Proposition~\ref{nuclear2} gives that $L(\lambda) =0$. Conversely, if $L(\lambda)=0$, then Bohr's theorem holds for $\lambda$
and $\mathcal{D}_{\infty}(\lambda)$ is by Theorem~\ref{equivalence} a $\lambda$-admissible Banach space.  Again Proposition~\ref{nuclear2} completes the proof.
\end{proof}

\begin{example}\text{}
\begin{enumerate}[(a)]
\item
$\mathcal{D}_{\infty,+}((n)) = H(\mathbb{D})$  is nuclear, since  $L((n)) =0$.
This is a classic (see e.g. \cite{meise1997introduction}).

\item
$\mathcal{D}_{\infty,+}((\log n))$ and $\mathcal{D}_{\infty,+}((\log p_{n}))$ are  both non-nuclear, since in both cases $L(\lambda)=1$. As mentioned before, the first example is due to Bonet \cite{bonet2018frechet}.

\item
$\mathcal{D}_{\infty,+}((\log n)^\alpha)$  is nuclear for $\alpha >  1$ (since $L(\lambda)=0$)
and not nuclear for $0<\alpha < 1$ (since $L(\lambda)=~\infty$).
\end{enumerate}
\end{example}

\subsection{Hardy spaces of Dirichlet series}  \label{hardy}

With the same spirit as in Section~\ref{structure} we apply now the abstract programme described in Section~\ref{Dseries} to the scale $\mathcal{H}_{p}(\lambda)$ of Hardy spaces of general Dirichlet series for $1 \leq p \leq \infty$ (see \cite{FVGaMeSe_20} for the case of ordinary series). Let us briefly observe that in this case the abscissa defined in \eqref{carapaz} now reads as
\begin{equation} \label{mcalain}
\sigma_{\mathcal{H}_{p}(\lambda)}(D) = \inf \big\{ \sigma\in \mathbb{R} \colon \sum a_{n} e^{- { \lambda_{n}\sigma}} e^{- \lambda_{n} s} \in \mathcal{H}_{p} (\lambda)  \big\} \,.
\end{equation}
With this, following \eqref{fx}, we consider the pre-Fr\'echet space
\[
\mathcal{H}_{p,+}(\lambda) = \big\{ D=\sum a_{n}e^{-\lambda_{n}s} \colon \sigma_{\mathcal{H}_{p}(\lambda)}(D)\leq 0 \big\} \,,
\]
endowed with the locally convex metrizable topology is generated by the sequence of norms
\[
\big\Vert \sum a_{n}e^{-\lambda_{n}s}\big\Vert_{\mathcal{H}_{p} (\lambda),k}
:=\Big\Vert \sum a_{n}e^{-\lambda_{n}\frac{1}{k}}e^{-\lambda_{n}s}\Big\Vert_{\mathcal{H}_{p}}\,,
\]
for $k\in \mathbb{N}$. As in \eqref{step}, for each $k \in \mathbb{N}$ we consider the canonically normed space defined by
\[
\mathcal{H}_{p,k} (\lambda) = \big\{ \sum a_{n} e^{- \lambda_{n} s} \in \mathfrak{D}(\lambda)
\colon \sum a_{n} e^{- \frac{ \lambda_{n}}{k}} e^{- \lambda_{n} s} \in \mathcal{H}_{p} (\lambda)\big\} \,,
\]
which (see \eqref{spec1} and \eqref{spec2}) leads to the two countable projective spectra
\begin{equation} \label{nizza}
\proj (\mathcal{H}_{p,k} (\lambda), i_k)_{k \in \mathbb{N}}
\,\,\, \text{ and } \,\,\,
\proj (\mathcal{H}_{p} (\lambda), \tau_k)_{k \in \mathbb{N}}\,.
\end{equation}
We turn now to the study of the structure of the spaces $\mathcal{H}_{p,+}(\lambda)$.

\subsubsection{Fr\'echet-Schwartz}

%
%

\begin{proposition} \label{projlimitH_p}
Let $\lambda$ be a frequency and $1 \leq p \leq \infty$.  Then  $\mathcal{H}_{p,+}(\lambda)$ is a Fr\'echet  Schwartz space which admits the following representations as projective limit
\[
\mathcal{H}_{p,+}(\lambda) = \proj (\mathcal{H}_{p,k} (\lambda), i_k) %
=\proj (\mathcal{H}_{p} (\lambda), \tau_k) \,.
\]
\end{proposition}
\begin{proof}
Proposition~\ref{projlimitX} and \eqref{nizza} give the two representations of the pre-Fr\'echet space $\mathcal{H}_{p,+}(\lambda)$ as projective limits. Since each $\mathcal{H}_{p}(\lambda)$ is
complete, the projective spectra in \eqref{nizza} consist of Banach spaces and, then, $\mathcal{H}_{p,+}(\lambda)$ is Fr\'echet. By Theorem~\ref{teo montel3} all translation operators
$\tau_k$ are compact operators; then Remark~\ref{proj1} gives that the space is also Schwartz.
\end{proof}

\begin{example} \label{galdos}
Let $\lambda$ be a $\mathbb{Q}$-linearly independent frequency. From \cite[Corollary~3.36]{defant2019hardy}
combined with Khinchin's inequality  we deduce that $\mathcal{H}_{p} (\lambda) = \ell_2$ (where each $\lambda$-Dirichlet series is identified with the sequence of its coefficients). As a consequence
$\mathcal{H}_{p,+} (\lambda)
=  \ell_{2}(A(\lambda))$ for every $1 \leq p < \infty$.
\end{example}

\subsubsection{Coincidence}

Since $\mathcal{D}_{\infty}(\lambda)\subset \mathcal{H}_{\infty}(\lambda)$, and $\mathcal{D}_{\infty,+}(\lambda) = \proj\mathcal{D}_{\infty,k}(\lambda)$
as well as  $\mathcal{H}_{\infty,+}(\lambda) = \proj\mathcal{H}_{\infty,k}(\lambda)$,  from Remark~\ref{proj} we obtain that
there is a continuous embedding
\[
 \mathcal{D}_{\infty,+}(\lambda) \,\hookrightarrow \, \mathcal{H}_{\infty,+}(\lambda)\,,
\]
that preserves  Dirichlet and Fourier coefficients.

\begin{theorem} \label{equalityHp}
$\mathcal{D}_{\infty,+}(\lambda)=\mathcal{H}_{\infty,+}(\lambda)$ if and only if Bohr's theorem holds for $\lambda$.
\end{theorem}
\begin{proof}
If Bohr's theorem holds for $\lambda$, then we know from Theorem~\ref{equivalence}  that $\mathcal{D}_{\infty,k}(\lambda)=\mathcal{H}_{\infty,k}(\lambda)$ for every $k$. Hence the claim follows by
Remark~\ref{proj}. Conversely, if $\mathcal{D}_{\infty,+}(\lambda)=\mathcal{H}_{\infty,+}(\lambda)$, then $\mathcal{D}_{\infty,+}(\lambda)$ is complete, and so we deduce from  Theorem~\ref{equi-1} that Bohr's theorem holds for $\lambda$.
\end{proof}

\subsubsection{Bases}

\begin{proposition} \label{basis}
Let  $\lambda=(\lambda_{n})$ be a frequency and $1 \leq p \leq \infty$.
Then the monomials $\{ e^{-\lambda_{n} s} \}_{n}$ form a  basis
\begin{enumerate}[(i)]
\item \label{basis1}  for $\mathcal{H}_{p,+}(\lambda)$,  whenever  $1 < p< \infty$.
\item \label{basis2} for $\mathcal{H}_{1,+}(\lambda)$,  whenever $\lambda$ satisfies Bohr's theorem.
\item \label{basis3} for $\mathcal{H}_{\infty,+}(\lambda)$ if and only if $\lambda$ satisfies Bohr's theorem.
\end{enumerate}
\end{proposition}

\noindent We structure the proof with the following lemma.

\begin{lemma} \label{0=}
For every frequency $\lambda$ and $1 < p < \infty$ we have
$\mathcal{H}_{p}(\lambda) = \mathcal{H}^0_{p}(\lambda)$, and for $p=1$ this holds true whenever $\lambda$ satisfies Bohr's theorem.
\end{lemma}

\begin{proof}
As we have already mentioned, by \cite[Theorem~4.16]{defant2019hardy} the monomials form a basis in $\mathcal{H}_p(\lambda)$. This settles the case $1 < p < \infty$. In order to tackle the case
$p=1$, let us recall first that, by definition,  $\mathcal{H}_{1}^{0}(\lambda) \subseteq \mathcal{H}_{1} (\lambda)$. We have to see that the reverse inequality holds if $\lambda$ satisfies
Bohr's theorem. We go for a moment into the theory of vector valued general Dirichlet series. The basic definitions needed here are just straightforward translations of the scalar valued ones. The reader is referred to \cite{CaDeMaSc_VV} for a
complete account on the theory. In \cite[Lemma~4.9]{defantschoolmann2019Hptheory} we have that, for any frequency $\lambda$, the mapping $\mathcal{H}_{1}(\lambda)\hookrightarrow \mathcal{D}_{\infty}(\lambda,\mathcal{H}_{1}(\lambda))$ given by
\begin{equation}\label{inclusion}
\sum a_{n} e^{-\lambda_{n}s} \mapsto \sum (a_{n}e^{-\lambda_{n} z})e^{-\lambda_{n}s}
\end{equation}
defines an isometry.
Once we have this, note that, given a Banach space $X$ and $\varepsilon >0$, there exists $c>0$ so that
\begin{equation*}
\sup_{N \in \mathbb{N}} \sup_{t \in \mathbb{R}} \Big\Vert \sum_{n=1}^{N} a_{n} e^{- \lambda_{n} (\varepsilon + it)}  \Big\Vert_{X}
\leq c \Big\Vert \sum a_{n} e^{- \lambda_{n} s} \Big\Vert_{ \mathcal{D}_{\infty} (\lambda, X)}
\end{equation*}
for every $X$-valued Dirichlet series in $\mathcal{D}_{\infty} (\lambda, X)$ (see \cite[Proof of Theorem~4.12]{CaDeMaSc_VV}). On the other hand, an argument with the Hahn-Banach theorem after \cite[Comment after Proposition~2.4]{schoolmann2018bohr} shows that
\[
\sup_{t \in \mathbb{R}} \Big\Vert \sum_{n=1}^{N} a_{n} e^{- \lambda_{n} (\varepsilon + it)}  \Big\Vert_{X}
= \sup_{\re s > \varepsilon} \Big\Vert \sum_{n=1}^{N} a_{n} e^{- \lambda_{n} s}  \Big\Vert_{X}
\]
for every $X$-valued Dirichlet polynomial. With this and \eqref{inclusion}, given $\sum a_{n} e^{- \lambda_{n}s} \in \mathcal{H}_{1}(\lambda)$ and $\varepsilon >0$ we have
\begin{multline*}
\Big\Vert \sum_{n=1}^{N} a_{n} e^{- \varepsilon \lambda_{n}} e^{-\lambda_{n} s} \Big\Vert_{ \mathcal{H}_{1} (\lambda)}
= \Big\Vert \sum_{n=1}^{N} \big( a_{n} e^{- \varepsilon \lambda_{n}} e^{-\lambda_{n} s} \big) e^{-\lambda_{n} z} \Big\Vert_{ \mathcal{D}_{\infty} (\lambda,\mathcal{H}_{1} (\lambda))} \\
= \sup_{\re z >0}  \Big\Vert \sum_{n=1}^{N} \big( a_{n} e^{-\lambda_{n} s} \big)  e^{- \varepsilon \lambda_{n}} e^{-\lambda_{n} z} \Big\Vert_{\mathcal{H}_{1} (\lambda)}
\leq c  \Big\Vert \sum \big( a_{n} e^{-\lambda_{n} s} \big)  e^{-\lambda_{n} z} \Big\Vert_{ \mathcal{D}_{\infty} (\lambda,\mathcal{H}_{1} (\lambda))}\\
= c  \Big\Vert \sum a_{n} e^{-\lambda_{n} s} \Big\Vert_{\mathcal{H}_{1} (\lambda)} \,.
\end{multline*}
Then Proposition~\ref{Prop:BasicEstimationsAbscissae} gives $\sigma_{\mathcal{H}_{1}^{0}(\lambda)}(D) \leq 0$ and this implies that $\big(\sum_{n=1}^{N}  a_{n} e^{-\lambda_{n} \sigma} e^{-\lambda_{n} s} \big)_{N}$ is  convergent for every $\sigma >0$. Since the series is in $\mathcal{H}_{1}(\lambda)$ we finally obtain  $\sum a_{n} e^{- \lambda_{n}s} \in \mathcal{H}_{1}^{0}(\lambda)$.
\end{proof}

\begin{proof}[Proof of Proposition~\ref{basis}]\text{}
Both statements \ref{basis1} and \ref{basis2} are immediate consequences of Proposition~\ref{basisA} and Lemma~\ref{0=}.
If the monomials $\{ e^{-\lambda_{n} s} \}$ form a  basis for $\mathcal{H}_{\infty,+}(\lambda)$, then they form a basis for its subspace
$\mathcal{D}_{\infty,+}(\lambda)$, and so the claim follows from Theorem~\ref{equivalence2} and Theorem~\ref{equalityHp}.
\end{proof}

We finish this section by making a short comment on the abscissas that we have defined in \eqref{mcalain}. For ordinary Dirichlet series (i.e. $\lambda = (\log n)$) we know from  \cite[Theorem~12.4]{defant2018Dirichlet} that the abscissa for any $1 \leq p \leq \infty$ can be reformulated as
\[
\sigma_{\mathcal{H}_{p}}(D) =
\inf \Big\{ \sigma >0 \colon \big( \sum_{n=1}^{N} \frac{a_{n}}{n^{\sigma}}  n^{-s} \big)_{N} \text{ converges in } \mathcal{H}_{p}((\log n))  \Big\} \,.
\]
With the notation from Section \ref{Abscissassection} this means that  $1 \leq p \leq \infty$ and any ordinary Dirichlet series $D$ we have
\[
\sigma_{\mathcal{H}_{p}((\log n))}(D) = \sigma_{\mathcal{H}^0_{p}((\log n))}(D)\,.
\]
Then Lemma~\ref{0=} shows that this holds for  $1 < p < \infty$ and any frequency $\lambda$, and for $p=1$ and
any frequency $\lambda$ satisfying Bohr's theorem. Finally, we note that under Bohr's theorem for $\lambda$
we by definition and Theorem~\ref{equivalence} also have that
\[
\sigma_{\mathcal{H}_{\infty}(\lambda)}(D) = \sigma_{\mathcal{D}_{\infty}(\lambda)}(D) =
\sigma_{\mathcal{D}^0_{\infty}(\lambda)}(D)=\sigma_{\mathcal{H}^0_{\infty}(\lambda)}(D)\,.
\]

\subsubsection{Nuclearity}

In Section~\ref{tour} we settled the question of when  $\mathcal{D}_{\infty,+}(\lambda)$ is nuclear. We face now the same question for the Fr\'echet spaces $\mathcal{H}_{p,+}(\lambda)$. Again, the answer comes from Proposition~\ref{nuclear2}. We
have already seen (in Propositions~\ref{projlimitH_p} and~\ref{basis}) that for $1 <  p < \infty$ and any frequency $\lambda$, the space $\mathcal{H}_{p} (\lambda)$ is complete and the monomials form a  basis. This is also the case for $p=1$
or $p=\infty$ whenever $\lambda$ satisfies Bohr's theorem. Let us finally recall that, if $L(\lambda)=0$, then Bohr's theorem holds for $\lambda$. With this altogether (and Proposition~\ref{nuclear2}) we have a full description of when are these Hardy
spaces nuclear.

\begin{proposition} \label{nuclearA}
Let  $\lambda$ be a frequency. Then
\begin{enumerate}[(i)]
\item for $1 < p< \infty$ the Fr\'echet space  $\mathcal{H}_{p,+}(\lambda)$ is nuclear if and only if  $L(\lambda)=0$.
\item for $p=1$ and $p=\infty$ the Fr\'echet space $\mathcal{H}_{p,+}(\lambda)$ is nuclear and $\lambda$ satisfies Bohr's theorem if and only if $L(\lambda)=0$.
\end{enumerate}
\end{proposition}

As we already pointed out in Section~\ref{sec:nuc}, if $L(\lambda)=0$ then all $\mathcal{H}_{p,+}(\lambda)$ coincide for $1 < p< \infty$.

%
%
%
%
%

\subsubsection{Translation}

An important fact within the theory of Hardy spaces of ordinary Dirichlet series is that the translation operator  $\tau_{\sigma}$,  defined for each $\sigma > 0$ as
\[
\tau_{\sigma} \big( \textstyle\sum a_n n^{-s} \big) = \displaystyle\sum \frac{a_n}{n^{\sigma}} n^{- s}\,,
\]
for every $1 \leq p < q < \infty$ is bounded as an operator from $\mathcal{H}_{p} = \mathcal{H}_{p}((\log n))$ into $\mathcal{H}_{q} = \mathcal{H}_{q}((\log n))$. This has as an immediate consequence that $\sigma_{\mathcal{H}_{p}}(D)= \sigma_{\mathcal{H}_{q}}(D)$ for every $1 \leq p,q < \infty$
and $D \in \mathfrak{D}((\log n))$
(see \cite[Chapter~12]{defant2018Dirichlet} for more details). Then we obtain as an immediate consequence that $\mathcal{H}_{p,+} = \mathcal{H}_{q,+}$ for every $1 \leq p,q, < \infty$, and, there is only one such space, denoted $\mathcal{H}_{+}$, which can be taken as $\mathcal{H}_{2,+}$ (see \cite{FVGaMeSe_20}). \\

We address now an analogous question for general Dirichlet series. For $\sigma \in \mathbb{R}$ we define the translation operator as
\[
\tau_{\sigma} \big( \textstyle\sum a_n e^{-\lambda_n s} \big) = \displaystyle\sum a_n e^{-\lambda_n \sigma} e^{-\lambda_n s}\,.
\]
Then we say that the frequency $\lambda$ is hypercontractive (for the translation operator) if, for every $\sigma >0$, the operator $\tau_{\sigma}: \mathcal{H}_{p} (\lambda) \to \mathcal{H}_{q}(\lambda)$ is bounded for every $1 \leq p \leq q < \infty$.

\begin{remark} \label{zum}
 It is obvious that a given  frequency $\lambda$ is   hypercontractive if and only if $\sigma_{\mathcal{H}_{p}(\lambda)}(D) = \sigma_{\mathcal{H}_{q}(\lambda)}(D)$
for every $D \in \mathfrak{D}(\lambda)$ and $1 \leq  p < q < \infty$, if and only if $\mathcal{H}_{p,+}(\lambda)= \mathcal{H}_{q,+}(\lambda)$ for every choice of  $1 \leq  p,q < \infty$.

In \cite{bayart_hyper} it is shown that there exist non-hypercontractive frequencies. More precisely, there is  a frequency $\lambda$ satisfying Bohr's condition  so that $\tau_{\sigma} : \mathcal{H}_1(\lambda) \to \mathcal{H}_2(\lambda)$
is not bounded for every $\sigma >0$.  In particular, $\mathcal{H}_{2,+}(\lambda) \varsubsetneqq \mathcal{H}_{1,+}(\lambda)$.
\end{remark}

Our aim now is to find conditions that imply that the frequency is hypercontractive for the translation operator.

\begin{remark} \label{zum2}
If $L(\lambda) =0$ or $\lambda$ is $\mathbb{Q}$-linearly indedependent, then $\lambda$ is hypercontractive.
Indeed, in both cases we  by Proposition~\ref{nuclear2}  and Example~\ref{galdos} know that all Fr\'echet spaces $\mathcal{H}_{p,+}(\lambda)$
coincide (as sequence spaces).
\end{remark}

We recall  that for each $0 < \eta < 1$ and $1 \leq p \leq q < \infty$ there is a bounded operator
$T_{\eta} : H_{p}(\mathbb{T}) \to H_{q}(\mathbb{T})$
such that
\[
T_{\eta} \Big( \sum_{k=0}^{n} c_{k} z^{k} \Big) = \sum_{k=0}^{n} c_{k} (\eta z)^{k}
\]
 Furthermore, $\Vert T_{\eta} \Vert \leq 1$ for every $\eta < \sqrt{p/q}$ (see e.g. \cite[Proposition~8.11]{defant2018Dirichlet}). For $N \in \mathbb{N} \cup \{\infty\}$ we know  from \cite[Theorem~12.10]{defant2018Dirichlet} that, if
 $\eta = (\eta_{k})_{1 \leq k \leq N} \subseteq (0,1)$ is such that $\sup_{n} \prod_{k=1}^{n} \Vert T_{\eta_{k}} \Vert < \infty$ (note that this is trivially satisfied if $N$ is finite), then there exists an operator
\begin{equation} \label{erloschte}
T_{\eta} : H_{p}(\mathbb{T}^{N}) \to H_{q}(\mathbb{T}^{N})
\end{equation}
so that
\begin{equation} \label{schar}
T_{\eta} \Big( \sum_{\alpha \in F \atop F \text{finite}} c_{\alpha} z^{\alpha} \Big) = \sum_{\alpha \in F \atop F \text{finite}} c_{\alpha} (\eta z)^{\alpha} \,,
\end{equation}
and $\Vert T_{\eta} \Vert \leq \sup_{n} \prod_{k=1}^{n} \Vert T_{\eta_{k}} \Vert$. If $\Lambda \subseteq \mathbb{N}_{0}^{N}$ (if $N = \infty$ this should be understood as $\mathbb{N}_{0}^{(\mathbb{N})}$)
we consider
\[
H_{p}^{\Lambda} (\mathbb{T}^{N}) = \{ f \in  H_{p} (\mathbb{T}^{N}) \colon \widehat{f}(\alpha)  \neq 0 \curvearrowright \alpha \in \Lambda  \} \,,
\]
which as a closed subspace of $H_{p} (\mathbb{T}^{N})$ is again a Banach space. A straightforward argument using \eqref{schar} and the density in $H_{p}^{\Lambda} (\mathbb{T}^{N})$ of the trigonometric polynomials with coefficients indexed on $\Lambda$ (see \cite[Theorem~3.14]{defantschoolmann2019Hptheory}) gives
\begin{equation} \label{himmel}
T_{\eta}  \big( H_{p}^{\Lambda} (\mathbb{T}^{N}) \big) \subseteq H_{q}^{\Lambda} (\mathbb{T}^{N}) \,
\end{equation}
and
\[
\Vert T_{\eta} : H_{p}^{\Lambda}(\mathbb{T}^{N}) \to H_{q}^{\Lambda}(\mathbb{T}^{N}) \Vert
\leq \Vert T_{\eta} : H_{p}(\mathbb{T}^{N}) \to H_{q}(\mathbb{T}^{N}) \Vert \,.
\]

\bigskip

Let us recall from Section~\ref{sec.montel} (see the comments preceding Corollary~\ref{abschnitt}) that, given a frequency $\lambda = (\lambda_{n})_{n \in \mathbb{N}}$,
there is a decomposition $\lambda = (R,B)$,
where $B=(b_{j})_{1 \leq j \leq  N}$ (for $N\in \mathbb{N}$ or $N = \infty$) is the basis and
$R=(r_{j}^{n})_{1 \leq j \leq  N}^n$ the Bohr matrix of $\lambda$.
A frequency $\lambda$ is said to be of natural type if each entry of $R$ is in $\mathbb{N}_{0}$, and  in this case  each row $\alpha$ of $R$ (we write $\alpha \in R$) may be considered as a finite
sequence in $\mathbb{N}_0^{(\mathbb{N})}$ (so $R \subset \mathbb{N}_0^{(\mathbb{N})}$).\\
Given a frequency  $\lambda$ of natural type, the Bohr transform $\mathfrak{B}$ defines an isometric isomorphism between $\mathcal{H}_{p}(\lambda)$ and $H_{p}^{R}(\mathbb{T}^{N})$.  More precisely, there is a unique onto isometry
\[
\mathfrak{B}: \mathcal{H}_{p}(\lambda) \to H_{p}^{R}(\mathbb{T}^{N})
\]
such for each $\alpha \in R$ and $n \in \mathbb{N}$ with $\lambda_n = \sum \alpha_j b_j$ we have that $\widehat{f}(\alpha) = a_n$
for all $D = \sum a_n e^{-\lambda_n s} \in \mathcal{H}_{p}(\lambda)$
and $f \in H_{p}^{R}(\mathbb{T}^{N})$ with $f= \psi(D)$ (see \cite[Theorem~3.31]{defantschoolmann2019Hptheory}).

\begin{theorem} \label{jubiliert}
Let $\lambda$ be a frequency with a decomposition $(B,R)$ of natural type so that $b_{j} >0$ for every $j$ and (if $B$ is infinite) $\lim_{j} b_{j} = \infty$.
Then, $\lambda$ is hypercontractive.
\end{theorem}
\begin{proof}
Fix some $1 \leq p \leq q < \infty$ and $\sigma >0$ and let us define $\eta_{j} = e^{- b_{j} \sigma}$ for each $1 \leq j \leq N$. Since all $b_{j}$ are positive, we have $0 < \eta_{j} < 1$ for every $j$. If $B$ is finite and has length $N$, then by \eqref{erloschte} and \eqref{himmel} we have a continuous operator
\[
T_{\eta} : H_{p}^{R}(\mathbb{T}^{N}) \to H_{q}^{R}(\mathbb{T}^{N}) \,.
\]
If $B$ is infinite, the fact that $\lim_{j} b_{j} = \infty$ implies that we  find some $j_{0}$ so that $\eta_{j} < \sqrt{p/q}$ for every $j \geq j_{0}$. Then $\sup_{n} \prod_{j=1}^{n} \Vert T_{\eta_{j}} \Vert \leq \prod_{j=1}^{j_{0}} \Vert T_{\eta_{j}} \Vert$ and we have a bounded operator $T_{\eta} : H_{p}(\mathbb{T}^{\infty}) \to H_{q}(\mathbb{T}^{\infty})$. This and \eqref{himmel} again gives
\[
T_{\eta} : H_{p}^{R}(\mathbb{T}^{\infty}) \to H_{q}^{R}(\mathbb{T}^{\infty}) \,.
\]
We now consider the bounded operator $\tau_{\sigma} = \mathfrak{B}^{-1} \circ T_{\eta} \circ \mathfrak{B} : \mathcal{H}_{p} (\lambda) \to \mathcal{H}_{q} (\lambda)$. Let us see that this is exactly the translation operator that we are looking for. To do this we look first at Dirichlet polynomials. Given $\sum_{n=1}^{k} a_{n} e^{-\lambda_{n} s}$ we write $c_{\alpha} = a_{n}$ if $\sum_{j} \alpha_{j} b_{j} = \lambda_{n}$ and  have
\[
\mathfrak{B} \Big( \sum_{n=1}^{k} a_{n} e^{-\lambda_{n} s} \Big)
= \sum_{\alpha \in R \atop \lambda_{1} \leq \sum_{j} \alpha_{j} b_{j} \leq \lambda_{k}} c_{\alpha} z^{\alpha} \,.
\]
The latter is a finite sum, and \eqref{schar} gives
\begin{multline*}
\mathfrak{B}^{-1} \circ T_{\eta} \Big( \sum_{\alpha \in R \atop \lambda_{1} \leq \sum_{j} \alpha_{j} b_{j} \leq \lambda_{k}} c_{\alpha} z^{\alpha} \Big)
= \mathfrak{B}^{-1} \Big( \sum_{\alpha \in R \atop \lambda_{1} \leq \sum_{j} \alpha_{j} b_{j} \leq \lambda_{k}} c_{\alpha} (\eta z)^{\alpha} \Big) \\
= \mathfrak{B}^{-1} \Big( \sum_{\alpha \in R \atop \lambda_{1} \leq \sum_{j} \alpha_{j} b_{j} \leq \lambda_{k}} c_{\alpha} e^{ - \sum_{j} b_{j} \alpha_{j} \sigma}  z^{\alpha} \Big)
= \sum_{n=1}^{k} a_{n} e^{-\lambda_{n} \sigma} e^{-\lambda_{n} s} \,.
\end{multline*}
This yields our claim for Dirichlet polynomials, but these are dense in $\mathcal{H}_{p}(\lambda)$ (see \cite[Theorem~3.26]{defantschoolmann2019Hptheory}). A standard argument using the density and the fact that convergence in $\mathcal{H}_{p}(\lambda)$ implies convergence (in $\mathbb{C}$) of the coefficients completes the proof.
\end{proof}

\begin{remark} \label{rejoice}
The frequencies $(\log n)$ and $(n)$ trivially satisfy the conditions in Theorem~\ref{jubiliert}. Also, if the frequency $\lambda$ is $\mathbb{Q}$-linearly independent (as, for example $(\log p_{n})$, being $(p_{n})$ the sequence of prime numbers), then one can just take $B=\lambda$ and $R$ given by $r_{j}^{n}=\delta_{j,n}$ to see that it satisfies the conditions in Theorem~\ref{jubiliert}.
As a straightforward consequence (see Remark~\ref{zum}), for each of these frequencies all the spaces $\mathcal{H}_{p,+}(\lambda)$ (with $1 \leq p < \infty$) are all isomorphic to each other as Fr\'echet spaces.
\end{remark}

\subsection{Almost periodic functions}

As we already pointed out in Section~\ref{prelim}, general Dirichlet series and uniformly almost functions are closely related. More precisely, we know from  \cite[Theorem~2.16]{defant2020riesz} (see also Theorem~\ref{equivalence2}) that there is an
isomorphism preserving Bohr and Dirichlet coefficients so that $\HRe = \mathcal{H}_{\infty} (\lambda)$. Our aim now is to find an analogous description for  $\mathcal{H}_{\infty,+}(\lambda)$. The first step is to find the proper space of almost periodic
functions, and to endow it with a convenient locally convex topology.
We denote by
\[
\HRep
\]
the space of all holomorphic functions
$f: [\re  >0] \to \mathbb{C}$ which are uniformly almost periodic on each abscissa $[\re = \sigma]$ and such that the $x$th Bohr coefficients of $f$ (recall \eqref{Bohrcoeffholo})
vanishes, whenever $x \notin \{ \lambda_n\mid n \in \mathbb{N}\}$.  Each such function is then bounded on every half plane $[\re > \varepsilon]$ (see \cite[Chapter~III, \S~3]{Be54}), and hence we may  endow $\HRep$  with the Fr\'echet topology given by the family of norms
\begin{equation} \label{normk}
\Vert f \Vert_{\infty,k} = \sup_{\re s > \frac{1}{k}} \vert f(s) \vert \,.
\end{equation}

\subsubsection{Projective description}

Again it is convenient to find proper projective descriptions of $\HRep$.
Consider first for each $k$ the Banach space
\[
\HRek{k} = \Big\{ f: [\re  > 1/k] \to \mathbb{C} \text{ holomorphic } \colon f ( \punkt + \tfrac{1}{k} ) \in \HRe  \Big\}
\]
endowed with the norm defined in \eqref{normk}. Then we get the projective spectrum
\[
\big(\HRek{k}, i_k\big)_{k \in \mathbb{N}}\,,
\]
where  the linking maps are  the restrictions
\[
i_k: \HRek{(k+1)} \hookrightarrow \HRek{k} \text{ given by } f\mapsto f|_{[\re >1/k]} \,.
\]
Given $f \in \HRek{k}$ the Bohr coefficients of $f$ (recall once again ~\eqref{Bohrcoeffholo}) are
\[
a_{\lambda_{n}}(f)=\lim_{T \to \infty} \frac{1}{2T} \int_{-T}^{T} f(\sigma+it)e^{(\sigma+it)\lambda_{n}} dt \,,
\]
where $\sigma>1/k$ is arbitrary (and the definition is independent of the chosen $\sigma$). Observe that with this definition the Bohr coefficients of $f\in \HRek{(k+1)}$ and $i_{k}(f)\in \HRek{k} $ coincide.\\

As in Section~\ref{baronnoir} we have a second possible projective spectrum that serves our purposes. To begin with, note that for each $k$, the mapping
\begin{equation*}
\varphi_k: \HRek{k}\to\HRe \, \text{ defind by } \,f \mapsto  f \big( \punkt + 1/k \big)
\end{equation*}
is an isometric bijection, where the inverse is given by $\varphi_{k}^{-1}(f)=f \big( \punkt - 1/k \big)$. Then we can consider the projective spectrum
\begin{equation*}
\big(\HRe, \tau_k  \big)_{k \in \mathbb{N}}\,,
\end{equation*}
where $\tau_k: \HRe \to \HRe$ is defined by
\[
f \mapsto f\big(\punkt + 1/k-1/(k+1)\big)\,.
\]

\medskip

\begin{proposition} \label{projlimit2}
Let $\lambda$ be a frequency. Then  $\HRe$ is a Fr\'echet-Schwartz space. Also, the mappings
\[
\HRep =  \proj\big(\HRek{k}, i_k\big) \, \text{ given by } \,
f \mapsto (f|_{[\re >1/k]})_{k=1}^\infty
\]
and
\[
\HRep =  \proj\big(\HRe, \tau_k\big)
 \, \text{ given by } \, f \mapsto \big(f ( \punkt + 1/k )\big)_{k=1}^\infty
 \]
are isomorphisms of Fr\'echet spaces.
\end{proposition}
\begin{proof}
Both projective descriptions follow exactly as in the proof of Proposition~\ref{projlimitX}.
In particular, looking at the second one  and taking into account that all spaces $\HRe$ are Banach, we
deduce from Theorem~\ref{teo montel2} that $\tau_{k}$ is compact for every $k$ and therefore $\HRep$ is a Fr\'echet-Schwartz space.
\end{proof}

\subsubsection{Coincidence}

We are now ready to show that an isomorphism as in Theorem~\ref{equivalence2} identifying coefficients also exists between $ \HRep$ and $\mathcal{H}_{\infty,+}(\lambda)$.

\begin{theorem} \label{piazzola}
The identification
\[
 \HRep =  \mathcal{H}_{\infty,+}(\lambda) \, \text{ given by } \, f \mapsto \sum a_{\lambda_{n}}(f)e^{-\lambda_{n}s},
\]
is a coefficient preserving isomorphism of Fr\'echet spaces.
\end{theorem}
\begin{proof}
We begin by seeing that for each fixed $k$ the mapping
\[
S_{k}:  \HRek{k} \to \mathcal{H}_{\infty,k}(\lambda)
\]
defined by
\[
f \mapsto \sum a_{\lambda_n}(f)e^{-\lambda_n s}
\]
is an isometric bijection. Take any $f\in  \HRek{k}$ and observe that the function $g:=f(\punkt+1/k)$ belongs to $\HRe$ and has Bohr coefficients
\[
a_{\lambda_{n}}(g)=a_{\lambda_{n}}(f)e^{-\frac{\lambda_{n}}{k}}\,,
\]
for $n \in \mathbb{N}$. Hence by Theorem~\ref{equivalence2} the  Dirichlet series $\sum a_{\lambda_n}(f)e^{-\frac{\lambda_{n}}{k}}e^{-\lambda_{n}s}=\sum a_{\lambda_n}(g)e^{-\lambda_{n}s}$ belongs to $\mathcal{H}_{\infty}(\lambda)$, and so $\sum a_{\lambda_n}(f)e^{-\lambda_n s}  \in \mathcal{H}_{\infty,k}(\lambda)$ with
\[
\big\|\sum a_{\lambda_n}(f)e^{-\lambda_n s}  \big\|_{ \mathcal{H}_{\infty}(\lambda),k}
=\big\| a_{\lambda_n}(f)e^{-\frac{\lambda_{n}}{k}}e^{-\lambda_{n}s} \big\|_{ \mathcal{H}_{\infty}(\lambda)}
=\|g\|_{\infty}= \|f\|_{\infty,k}\,.
\]
This shows that $S_{k}$ is a well defined isometry. Conversely, if $\sum a_{n}e^{-\lambda_{n}s}\in \mathcal{H}_{\infty,k}(\lambda)$, then by definition and again  Theorem~\ref{equivalence2} we can
find some $g\in \HRe$ such that $a_{\lambda_{n}}(g)=a_{n}e^{-\frac{\lambda_{n}}{k}}$ for all $n$. Now the function $f:=g(\punkt -1/k)$ belongs to $\HRek{k}$ and has Bohr coefficients
\[
a_{\lambda_{n}}(f)=e^{\frac{\lambda_{n}}{k}}a_{\lambda_{n}}(g)=a_{n},
\]
for $n\in \mathbb{N}$. This shows that $S_{k}$ is surjective and, hence, an isometric bijection. Now Remark~\ref{proj} implies that the mapping
\[
 S: \proj \big(\HRek{k}, i_k\big) \to \proj (\mathcal{H}_{\infty,k}\big(\lambda),i_k\big)\, \text{ given by } \, (f_k) \mapsto (S_k(f_k))
\]
is a Fr\'echet isomorphism. Moreover, if $\rho_m$ and $\pi_m$  denote the canonical projections of the respective projective spectra into   $\HRek{m}$ and $\mathcal{H}_{\infty,m}(\lambda)$, we have
\[
\pi_m \circ S = S_m \circ \rho_m\,.
\]
Using the projective descriptions of the spaces given in Proposition~\ref{projlimit} and~\ref{projlimit2}, this immediately gives that for each $f\in \HRep$, the $n$th Dirichlet coefficient of $S(f)$ equals $a_{\lambda_{n}}(f)$.
\end{proof}

By  Propositions~\ref{basis} and~\ref{nuclearA} we get the following corollary.

\begin{corollary} \text{}
\begin{enumerate}[(i)]
\item The monomials  $\{ e^{-\lambda_{n} s} \}_{n}$ form a basis in $\HRep$ if and only if $\lambda$ satisfies Bohr's theorem.
\item $\HRep$ is nuclear and $\lambda$ satisfies Bohr's theorem if and only if  $L(\lambda)=0$.
\end{enumerate}
\end{corollary}


\begin{thebibliography}{10}

\bibitem{amerioprouse}
L.~Amerio and G.~Prouse.
\newblock {\em Almost-periodic functions and functional equations}.
\newblock Van Nostrand Reinhold Co., New York-Toronto, Ont.-Melbourne, 1971.

\bibitem{bayart2002hardy}
F.~Bayart.
\newblock Hardy spaces of {D}irichlet series and their composition operators.
\newblock {\em Monatsh. Math.}, 136(3):203--236, 2002.

\bibitem{bayart_hyper}
F.~Bayart.
\newblock Personal communication, 2020.

\bibitem{Be54}
A.~S. Besicovitch.
\newblock {\em Almost periodic functions}.
\newblock Dover Publications, Inc., New York, 1955.

\bibitem{Bohr}
H.~Bohr.
\newblock \"{U}ber die gleichm\"{a}\ss ige {K}onvergenz {D}irichletscher
  {R}eihen.
\newblock {\em J. Reine Angew. Math.}, 143:203--211, 1913.

\bibitem{BohrBemerkungen}
H.~{Bohr}.
\newblock {Einige Bemerkungen \"uber das Konvergenzproblem Dirichletscher
  Reihen.}
\newblock {\em {Rend. Circ. Mat. Palermo}}, 37:1--16, 1914.

\bibitem{bonet2018frechet}
J.~Bonet.
\newblock The {F}r\'{e}chet {S}chwartz algebra of uniformly convergent
  {D}irichlet series.
\newblock {\em Proc. Edinb. Math. Soc. (2)}, 61(4):933--942, 2018.

\bibitem{CaDeMaSc_VV}
D.~Carando, A.~Defant, F.~Marceca, and I.~Schoolmann.
\newblock Vector-valued general {D}irichlet series.
\newblock {\em arXiv preprint arXiv:2001.09656}, 2020.

\bibitem{defant2018Dirichlet}
A.~Defant, D.~Garc\'{\i}a, M.~Maestre, and P.~Sevilla-Peris.
\newblock {\em Dirichlet {S}eries and {H}olomorphic {F}unctions in {H}igh
  {D}imensions}, volume~37 of {\em New Mathematical Monographs}.
\newblock Cambridge University Press, Cambridge, 2019.

\bibitem{andreaspabloantonio}
A.~Defant, A.~P\'{e}rez, and P.~Sevilla-Peris.
\newblock A note on abscissas of {D}irichlet series.
\newblock {\em Rev. R. Acad. Cienc. Exactas F\'{\i}s. Nat. Ser. A Mat. RACSAM},
  113(3):2639--2653, 2019.

\bibitem{defant2019hardy}
A.~{Defant} and I.~{Schoolmann}.
\newblock {Hardy spaces of general Dirichlet series --- a survey}.
\newblock In {\em {Function spaces XII. Selected papers based on the
  presentations at the 12th conference, Krakow, Poland, July 9--14, 2018}},
  pages 123--149. Warsaw: Polish Academy of Sciences, Institute of Mathematics,
  2019.

\bibitem{defantschoolmann2019Hptheory}
A.~Defant and I.~Schoolmann.
\newblock $\mathcal{H}_p$-theory of general {D}irichlet series.
\newblock {\em J. Fourier Anal. Appl.}, 25(6):3220--3258, 2019.

\bibitem{defant2020riesz}
A.~Defant and I.~Schoolmann.
\newblock Riesz means in {H}ardy spaces on {D}irichlet groups.
\newblock {\em Math. Ann.}, 378(1-2):57--96, 2020.

\bibitem{defant2020variants}
A.~Defant and I.~Schoolmann.
\newblock Variants of a theorem of {H}elson on general {D}irichlet series.
\newblock {\em J. Funct. Anal.}, 279(5):108569, 37, 2020.

\bibitem{FVGaMeSe_20}
T.~Fern\'{a}ndez~Vidal, D.~Galicer, M.~Mereb, and P.~Sevilla-Peris.
\newblock Hardy space of translated {D}irichlet series.
\newblock {\em arXiv preprint arXiv:2003.04041}, 2020.

\bibitem{FlWl68}
K.~Floret and J.~Wloka.
\newblock {\em Einf\"{u}hrung in die {T}heorie der lokalkonvexen {R}\"{a}ume}.
\newblock Lecture Notes in Mathematics, No. 56. Springer-Verlag, Berlin-New
  York, 1968.

\bibitem{hardy2013general}
G.~H. Hardy and M.~Riesz.
\newblock {\em The general theory of {D}irichlet's series}.
\newblock Cambridge Tracts in Mathematics and Mathematical Physics, No. 18.
  Stechert-Hafner, Inc., New York, 1964.

\bibitem{HLS}
H.~Hedenmalm, P.~Lindqvist, and K.~Seip.
\newblock A {H}ilbert space of {D}irichlet series and systems of dilated
  functions in {$L^2(0,1)$}.
\newblock {\em Duke Math. J.}, 86(1):1--37, 1997.

\bibitem{jarchow2012locally}
H.~Jarchow.
\newblock {\em Locally convex spaces}.
\newblock B. G. Teubner, Stuttgart, 1981.
\newblock Mathematische Leitf\"{a}den. [Mathematical Textbooks].

\bibitem{Landau}
E.~Landau.
\newblock \"{U}ber die gleichm\"{a}\ss ige {K}onvergenz {D}irichletscher
  {R}eihen.
\newblock {\em Math. Z.}, 11(3-4):317--318, 1921.

\bibitem{meise1997introduction}
R.~Meise and D.~Vogt.
\newblock {\em Introduction to functional analysis}, volume~2 of {\em Oxford
  Graduate Texts in Mathematics}.
\newblock The Clarendon Press, Oxford University Press, New York, 1997.

\bibitem{pereyraward}
M.~C. Pereyra and L.~A. Ward.
\newblock {\em Harmonic analysis}, volume~63 of {\em Student Mathematical
  Library}.
\newblock American Mathematical Society, Providence, RI; Institute for Advanced
  Study (IAS), Princeton, NJ, 2012.
\newblock From Fourier to wavelets, IAS/Park City Mathematical Subseries.

\bibitem{queffelec2013diophantine}
H.~Queff\'{e}lec and M.~Queff\'{e}lec.
\newblock {\em Diophantine approximation and {D}irichlet series}, volume~2 of
  {\em Harish-Chandra Research Institute Lecture Notes}.
\newblock Hindustan Book Agency, New Delhi, 2013.

\bibitem{rudin_groups}
W.~Rudin.
\newblock {\em Fourier analysis on groups}.
\newblock Wiley Classics Library. John Wiley \& Sons, Inc., New York, 1990.
\newblock Reprint of the 1962 original, A Wiley-Interscience Publication.

\bibitem{schoolmann2018bohrA}
I.~Schoolmann.
\newblock {\em Hardy spaces of general {D}irichlet series and their maximal
  inequalities}.
\newblock PhD thesis, Carl von Ossietzky University of Oldenburg, 2020.

\bibitem{schoolmann2018bohr}
I.~Schoolmann.
\newblock On {B}ohr's theorem for general {D}irichlet series.
\newblock {\em Math. Nachr.}, 293(8):1591--1612, 2020.

\end{thebibliography}
\end{document}